\documentclass{article}[12pt]
\usepackage{graphicx, mathtools}
\usepackage{amssymb, amsmath, amsthm, amsfonts}
\usepackage{bbm}
\usepackage{color}
\usepackage{tikz}
\usepackage{blindtext}
\usepackage{subcaption}
\usepackage{float}
\usepackage[colorinlistoftodos,prependcaption,textsize=tiny]{todonotes}
\usepackage{xargs} 
\usepackage{enumitem}
\usepackage{pgfplots}
\pgfplotsset{compat=1.18}

\newtheorem{theorem}{Theorem}[section]
\newtheorem{lemma}{Lemma}[section]
\newtheorem{proposition}{Proposition}[section]
\theoremstyle{definition}

\newtheorem{example}{Example}[section]

\newcommand{\R}{\mathbb{R}}
\newcommand{\N}{\mathbb{N}}
\newcommand{\Z}{\mathbb{Z}}
\newcommand{\pr}{\mathbb{P}}
\newcommand{\E}{\mathbb{E}}
\newcommand{\1}{\mathbbm{1}}
\newcommand{\dd }{\mathrm{d} }

\definecolor{vry}{RGB}{253, 231, 37}

\definecolor{vrg}{RGB}{94,201,98}

\definecolor{vrdg}{RGB}{33, 145, 140}

\definecolor{vrb}{RGB}{59,82,139}

\definecolor{vrp}{RGB}{68,1,84}

\definecolor{vro}{RGB}{249,142,9}
\definecolor{vrrp}{RGB}{126, 3, 167}

\definecolor{vrr}{RGB}{188,55,84}

\definecolor{vrnb}{RGB}{13,8,135}

\newcommandx{\unsure}[2][1=]{\todo[linecolor=vrr,backgroundcolor=vrr!15,bordercolor=vrr!25,#1]{#2}}
\newcommandx{\change}[2][1=]{\todo[linecolor=vrb,backgroundcolor=vrb!15,bordercolor=vrb!25,#1]{#2}}
\newcommandx{\info}[2][1=]{\todo[linecolor=vrdg,backgroundcolor=vrdg!15,bordercolor=vrdg!25,#1]{#2}}
\newcommandx{\improve}[2][1=]{\todo[linecolor=vrp,backgroundcolor=vrp!15,bordercolor=vrp!25,#1]{#2}}
\newcommandx{\thiswillnotshow}[2][1=]{\todo[disable,#1]{#2}}
\newcommandx{\add}[2][1=]{\todo[linecolor=vro,backgroundcolor=vro!15,bordercolor=vro!25,#1]{#2}}

\usepackage[colorlinks=true]{hyperref}
\hypersetup{colorlinks=true, linkcolor=vrr, citecolor=vrr,  filecolor=blue,urlcolor=vrb}

\begin{document}
\title{Annealed Survival Probability of Random Walk in an Inhomogeneous Poisson Environment of Mobile Traps}
\date{}

\author{Pradeeptha R Jain \\
  {\small \it International Centre for Theoretical Sciences (ICTS) - TIFR, Bengaluru, India.}}

\maketitle

\begin{abstract}
     We study the annealed survival probability of a random walk in an inhomogeneous Poisson environment of mobile traps on $\mathbb{Z}^d$. In dimensions $d=1,2$, we determine the asymptotics for the annealed survival probability of the walker under suitable  assumptions on the average trap intensity and study how inhomogeneity in the initial trap configuration affects these asymptotics. Our results extend the asymptotics proved for the homogeneous setting in \cite{DGRS2012}. We also establish a law of large numbers, central limit theorem and large deviation principle for the inhomogeneous Poisson trap environment in $d\ge 1$, generalising the corresponding results proved for the homogeneous case considered in \cite{CG1984}. We present a class of examples of inhomogeneous trap environments where the decay rate of survival probability can be identified and also instances where it does not decay with time.
\end{abstract}

\textit{AMS subject classification:} 60K37, 60K35, 82C22.

\textit{Keywords:} Pascal principle, inhomogeneous Poisson trap environment, trapping dynamics, cumulant expansion.

\section{Introduction}

The diffusion of a particle in space ($\Z^d$ or $\R^d$) among randomly placed traps, has been widely studied in the mathematics and physics literature. Some models include transport processes in disordered media, directed polymers in random environments, and branching random walks in random environments. 
We refer the reader to the recent review article \cite{ADS2019} and \cite{DGRS2012} for a detailed survey of the literature. 

When the traps are immobile the problem is well studied and there is extensive literature (see monographs \cite{S1998, K2016} and references therein). The first mathematical results were obtained in the continuum setting, for the problem of Brownian motion among Poissonian obstacles, where the traps are balls whose centers are placed according to a homogeneous Poisson point process on $\R^d$. The Brownian motion starts at the origin and is annihilated at a rate $\gamma$ times the number of traps it is contained in. Using large deviation techniques, Donsker and Varadhan~\cite{DV1975} showed that the annealed survival probability of the Brownian motion decays asymptotically as $\exp\{-C t^{\frac{d}{d+2}}\}$ while Sznitman~\cite{S1998} established that the corresponding quenched survival probability decays asymptotically as $\exp\{-\bar{C} \frac{t}{(\log t)^{2/d}}\}$
using the method of enlargement of obstacles. Similar results have also been obtained when the traps are placed according to an i.i.d. Bernoulli distribution \cite{A1995, B1994, DV1979} and for more general random trapping potentials in the framework of the parabolic Anderson model (see \cite{K2016} for more details).  

In contrast to the immobile traps, very little is known when the traps are mobile.  Redig~\cite{R1994} considered a trapping potential generated by a reversible Markov process, such as a homogeneous Poisson system of random walks, or the symmetric exclusion process in equilibrium, and obtained exponential upper bounds on the annealed survival probability using spectral techniques for the process of traps viewed from the random walk. In \cite{DGRS2012}, the precise asymptotics for the annealed survival probability of the random walk in $\Z^d$ was established when the mobile trapping environment was a homogeneous Poisson system of random walks. They showed that it decays 
asymptotically as $\exp\{-C t^{1/2}\}$ in $d=1$, $\exp\{-C' {\frac{t}{\log t}}\}$ in $d=2$ and exponentially in $d\ge 3$. They also showed that the corresponding quenched survival probability decays at a well defined exponential rate in all dimensions. For the same model, it was shown in \cite{ADS2017} that the annealed path of the random walk in $\Z$ (and Brownian motion in $\R$ \cite{Oz2019}) conditioned on survival until time $t$ is subdiffusive and an invariance principle was also established in \cite{ADS2025} for $d\ge 6$. The continuum analogue of the result in \cite{DGRS2012} was obtained in \cite{PSSS2013}. Further extentions to L\'evy trap model were carried out by \cite{DSS2014}. The model has also been studied in the physics literature, where in  \cite{MOBC2003, MOBC2004} the focus was on the annealed survival probability. 
Cox and Griffeath~\cite{CG1984}, following a question posed by F. Spitzer, initiated the study of occupation-time large deviations for the homogeneous Poisson system of independent random walks on $\mathbb{Z}^d$. Port~\cite{P1966, P1967} showed strong laws and central limit theorem for the same in discrete time. 

We consider the model where a random walker moves among an inhomogeneous Poisson system of moving traps on $\Z^d$. Initially at each site $y\in \Z^d$, we place Poisson$(\nu_y)$ number of traps (independently) for $\nu_y\in [0,\infty), y\in \Z^d$. Each trap independently performs random walk on $\Z^d$. The random walker is killed upon contact with a trap at a fixed rate $\gamma\in (0,\infty]$. We obtain the asymptotics of the annealed survival probability (see Section~\ref{subsec:model} for precise definition) for the above model in dimensions $d=1,2$ when $(\nu_y)_{y\in \Z^d}$  are bounded from above and on average the number of traps is minimum at the origin for all $t\ge 0$ (see Theorem~\ref{thm:annealed}).  We also prove law of large numbers, central limit theorem and large deviation results for the occupational time functionals: the number of traps at the origin and the number of distinct traps that visit the origin (see Theorem~\ref{thm:Dt} and Theorem~\ref{thm:Nt}).

The rest of the introduction is organised as follows. In Section~\ref{subsec:model}, we provide a precise description of the model and statement of the main results.
We present examples of inhomogeneous trap environments where our main results apply and instances where the annealed survival probability decays slower than the homogeneous setting considered in \cite{DGRS2012} in Section~\ref{subsec:examples}. We 
discuss the significance of our main results, state open problems and outline of the proofs of main results in Section~\ref{subsec:discuss}. 

\subsection{Model and Main Results}
\label{subsec:model}
We will now define the model precisely. Let $X := (X (t))_{t \geq 0}$ be a continuous time simple symmetric random walk  on $\mathbb{Z}^d$ with jump rate $\kappa \geq 0$. Let $(Y_j^y)_{1 \leq j \leq N_y, y \in \mathbb{Z}^d}$ be a collection of independent simple symmetric random walks on $\mathbb{Z}^d$ with jump rate 1. Here, $N_y$ is the number of walks starting from $y$ at time $0$, distributed independently as Poisson random variable with mean $\nu_y$, and $Y_j^y := (Y_j^y (t))_{t \geq 0}$ denotes the $j$-th walk that starts at $y$ at time $0$. The collection of walks $(Y_j^y)_{1 \leq j \leq N_y, y \in \mathbb{Z}^d}$ will act as mobile traps among which the random walk $X$ evolves. Denote the number of traps present at site $x$ at time $t$ by
\begin{equation}
  \xi (t, x) := \sum_{y \in {\mathbb{Z}^d}, 1 \leq j \leq N_y}
  \delta_x  (Y_j^y (t)). \label{randomfield}
\end{equation}
At each time $t\ge 0$, the walk $X$ is killed at rate $\gamma\xi(t,X(t))$.

One of the main quantity of interest in trapping problems is the probability that the particle survives until a given time. Conditional on the realization of the trap field $\xi$, the probability that the particle survives up to time $t$ is given by
$$Z_{t,\xi}^\gamma := \E_0^X\left[\exp\left\{-\gamma\int_0^t \xi(s,X(s))\,\dd s\right\}\right],
$$
where $\E^X_0$ denotes the expectation with respect to the random walk $X$ starting from the origin. The quantity $Z_{t,\xi}^\gamma$ is referred to as the \emph{quenched survival probability}. Averaging further over the randomness of the trap field $\xi$ yields the \emph{annealed survival probability},
\begin{equation}
    Z_t^\gamma :=\E^\xi[Z_{t,\xi}^\gamma] = \E^\xi\E_0^X\left[\exp\left\{-\gamma\int_0^t \xi(s,X(s))\,\dd s\right\}\right],
\label{eqn:annealedSP}
\end{equation}
where $\E^\xi$ denotes the expectation with respect to the trap field $\xi$. We begin by introducing some notation and assumptions used in the theorem statements, followed by the result on the asymptotics of the annealed survival probability.
It can be verified that, for every fixed $t\ge0$ and $x\in \Z^d$, the number of traps $\xi(t,x)$ is Poisson distributed with mean
\begin{equation}
m(t,x) := \E[\xi(t,x)]= \sum_{y\in\mathbb{Z}^d} \nu_y\,\pr_y^Y(Y(t)=x),
\label{eqn:poissonmean}
\end{equation}
where $\pr_y^Y$ denotes the law of a continuous-time simple symmetric random walk $Y$ with jump rate $1$, starting from site $y$ at time $0$. We impose the following assumptions on the trap field $\xi$.

\noindent \textbf{Assumptions:}
\begin{enumerate}[label=(A\arabic*)]
    \item \label{list:A1}The initial Poisson means $(\nu_y)_{y\in \mathbb{Z}^d}$ are bounded from above, i.e.,
$$
0< \nu^* := \sup_{y\in \mathbb{Z}^d} \nu_y < \infty.
$$
\item \label{list:A2} For all $x\in \Z^d$ and $t\ge 0$,
  \begin{equation}
      m(t,x)\ge m(t,0),
      \label{eqn:pascalcond} 
  \end{equation}
\end{enumerate}

Our first result establishes the precise asymptotics for the annealed survival probability in dimensions $d=1,2$ under assumptions \ref{list:A1} and \ref{list:A2}.
\begin{theorem}\label{thm:annealed}
    Assume $\gamma \in (0,\infty],\kappa \ge 0$, {\rm \ref{list:A1}} and {\rm \ref{list:A2}}. Then the following holds:
    \begin{equation}
        Z^{\gamma}_t = \begin{cases}
            \exp\left\{- \, \bar{m}\sqrt{\frac{8}{\pi}}t^{\frac{1}{2}}(1+o(1))\right\},& \text{if }d=1,\\[2mm]
        \exp\left\{- {\pi\,\bar{m}\frac{t}{\log t}(1+o(1))}\right\}, & \text{if } d=2.
        \end{cases}
        \label{eqn:annealed result}
    \end{equation}
    where $
     0< \bar{m}:=\lim\limits_{t\to \infty} \frac{1}{t}\int_0^t m(s,0)\, \dd s <\infty$.
\end{theorem}
Note that the leading-order asymptotics of $Z^\gamma_t$ in \eqref{eqn:annealed result} are independent of the killing rate $\gamma$. Further, when $\nu_y \equiv \nu>0$, we recover the corresponding annealed asymptotics of \cite[Theorem~1.1]{DGRS2012} for the homogeneous setting. We discuss several examples of trap environments in Section~\ref{subsec:examples} and discuss the proof strategy in Section~\ref{subsec:discuss}. 
We next analyse two occupational time functionals of the trap environment at the origin. Before we state our next theorems, we introduce an assumption which is weaker than \ref{list:A2}, namely
 
\textbf{Assumption:}
 \begin{enumerate}[label=(A\arabic*), start=3]
     \item \label{list:A3} There exist $\alpha=\alpha(d)$ with $\alpha(1)\in (1/2,1]$ and $\alpha(d) \in (0,1]$ in $d\ge 2$, a slowly varying function at infinity $L:(0,\infty)\to (0,\infty)$, and a constant $\bar m \in (0,\infty)$ such that
$$
\lim_{t\to\infty}\frac{1}{t^\alpha L(t)}\int_0^t m(s,0)\,\mathrm{d}s=\bar m.
$$
 \end{enumerate}
 In Lemma~\ref{lm:alpha1} we will show that \ref{list:A1} and \ref{list:A2} imply \ref{list:A3} with $\alpha =1$ and $L(\cdot) \equiv 1$.  We also need to setup the following notation. Let
\begin{equation}
    a_t = \begin{cases}
        \,t^{\alpha - \frac{1}{2}}L(t) & \text{ if } d=1,\\
        \,\frac{t^\alpha L(t)}{\log t} & \text{ if } d=2,\\
        \, t^{\alpha}L(t) &\text{ if } d\ge 3,
    \end{cases}
    \label{eqn:defat}
\end{equation}
$G_d(0) = \int_0^\infty p_t(0)\,\dd t$, where $p_t(\cdot)$ is the transition kernel of a rate 1 simple symmetric random walk on $\Z^d$ and $Z$ be a standard normal random variable.
 
 Our second result is on the number of visits to the origin by the traps until time $t$, given by.
\begin{equation}
    D_t(\xi) = \int_0^t \xi(s,0)\,\dd s
    \label{eqn:Dtdefn}
\end{equation}
Before stating the result, we introduce the rate function governing the large deviations for $D_t$, $I_D: (\beta_-, \infty) \to [0, \infty)$, given by
\begin{equation}
I_D(\beta) = \begin{cases}
    \beta \lambda_\beta - \Psi(\lambda_{\beta}) & \text{if } d=1,\\
    \pi\bar{m}(\sqrt{\beta}-1)^2 & \text{if } d=2,\\
    \frac{\bar{m}}{G_d(0)}(\sqrt{\beta}-1)^2 & \text{if } d \ge 3,
\end{cases}
\label{eqn:rateID}
\end{equation}
where $\beta_- = 0$ for $d =1$ and $d\ge 3$, $\beta_- = 1/4$ for $d=2$, and $\lambda_\beta$ is the unique solution to $\Psi'(\lambda_\beta) = \beta$ for the function
\begin{equation*}
\Psi(\lambda) = e^{\frac{\lambda^2}{2\bar{m}^2}}\lambda^{1-2\alpha}\int_0^\lambda e^{-\frac{\eta^2}{2\bar{m}^2}}\eta^{2\alpha-1}\left[2\alpha + \frac{\sqrt{2}\left(\alpha + \frac{1}{2}\right)\Gamma(\alpha+1)}{\bar{m}\Gamma\left(\alpha + \frac{3}{2}\right)}\eta\right]\mathrm{d}\eta, \quad \Psi(0)=0.
\end{equation*}
The following theorem establishes a law of large numbers, central limit theorem, and large deviation principle for $D_t$.\begin{theorem}\label{thm:Dt}
    Under Assumptions {\rm\ref{list:A1}} and {\rm \ref{list:A3}}, the following hold:
    \begin{enumerate}
        \item[(a)]  As $t\to \infty$, $$\frac{D_t}{\E^\xi[D_t]} \to 1 \quad \text{almost surely.}$$
        \item[(b)]  As $t\to \infty$, $$\frac{\sqrt{a_t}}{\sigma_D}\left(\frac{D_t}{\E^\xi[D_t]}-1\right)\xrightarrow{d}Z,$$ where 
        $$\sigma_D^2 = \begin{cases}
            \frac{\sqrt{2}\Gamma(\alpha +1)}{\bar{m}\Gamma(\alpha + \frac{3}{2})}& \text{ if } d=1,\\
            \frac{2}{\pi \bar{m}} & \text{ if } d=2,\\
            \frac{2G_d(0)}{\bar{m}} & \text{ if } d\ge 3.
        \end{cases}$$
        \item[(c)]  For any $\beta > 1$,
        $$
        \lim_{t \to \infty}\frac{1}{a_t}\log\mathbb{P}\left(\frac{D_t}{\mathbb{E}^\xi[D_t]}> \beta\right) = -I_D(\beta) \in (-\infty, 0),
        $$
        and for any $\beta \in (\beta_-, 1)$,
        $$
        \lim_{t \to \infty}\frac{1}{a_t}\log \mathbb{P}\left(\frac{D_t}{\mathbb{E}^\xi[D_t]}< \beta\right) = -I_D(\beta) \in (-\infty, 0).
        $$
    \end{enumerate}
\end{theorem}
Our third result is on the number of distinct particles which visit the origin up to time $t$, given by
\begin{equation}
    N_t(\xi) = \sum_{y\in \Z^d}\sum_{1\le j \le  N_y} \1_{\{Y^j_y(s)=0 \text{ for some }s\le t\}}.
\end{equation}
Before stating the result, we introduce the rate function for $N_t$,  $I_N: (0, \infty) \to [0, \infty)$, given by
\begin{equation}
I_N(\beta) = \theta\bar{m}(\beta\log \beta - \beta + 1),
\label{eqn:rateIN}
\end{equation}
where the constant $\theta$ is defined as
\begin{equation}
\theta = \begin{cases}
    \frac{\sqrt{2}\Gamma(\alpha+1)}{\Gamma(\alpha +\frac{1}{2})} & \text{if } d=1,\\
    \pi & \text{if } d=2,\\
    \frac{1}{G_d(0)} & \text{if } d \ge 3.
\end{cases}
\label{eqn:theta}
\end{equation}
 The following theorem establishes a law of large numbers, a central limit theorem, and a large deviation principle for $N_t$.

\begin{theorem}\label{thm:Nt}
       Under Assumptions {\rm\ref{list:A1}} and {\rm\ref{list:A3}}, the following hold:
        \begin{enumerate}[label=(\alph*)]
         \item  As $t \to \infty$, 
        $$\frac{N_t}{\mathbb{E}^\xi[N_t]} \to 1 \quad \text{almost surely.}$$
         \item  As $t \to \infty$,
        $$\sqrt{\theta\bar{m}a_t}\left( \frac{N_t}{\mathbb{E}^\xi[N_t]}-1\right) \xrightarrow{\mathrm{d}}Z.$$
        \item  For any $\beta > 1$,
        $$\lim_{t \to \infty}\frac{1}{a_t}\log\mathbb{P}\left( \frac{N_t}{\mathbb{E}^\xi[N_t]}> \beta\right) = -I_N(\beta) \in (-\infty, 0),$$
        and for any $\beta \in (0, 1)$,
        $$\lim_{t \to \infty}\frac{1}{a_t}\log \mathbb{P}\left( \frac{N_t}{\mathbb{E}^\xi[N_t]}<\beta\right) = -I_N(\beta) \in (-\infty, 0).$$
    \end{enumerate}
\end{theorem}
When $\nu_y \equiv \nu>0$, we recover the corresponding \cite[Theorem~1 and Theorem~2]{CG1984} for the homogeneous setting. 

\subsection{Examples}\label{subsec:examples}
We now present examples of the initial intensity profile $(\nu_y)_{y\in\mathbb{Z}^d}$ for which the assumptions of the main theorems hold as well as examples for which they do not. 
The first example considers a periodic environment, where the initial mean number of traps alternate between two values according to the parity of the lattice site. 
\begin{example}\label{eg:periodic}

  In this example, we first split the integer lattice into two sets $A$ and its complement $A^c$. Let $A$ be the set of all lattice sites whose coordinates sum to an odd number,\\ $$A = \left\{ x = (x_1, \ldots, x_d) \in \mathbb{Z}^d :\;
  \sum_{i = 1}^d x_i = 2 k + 1 \;\text{for}\;
  \text{some}\;k \in \mathbb{Z} \right\}.$$
  \begin{figure}[H]
        \centering
        \begin{subfigure}{.5\textwidth}
            \centering
            \includegraphics[width=8cm]{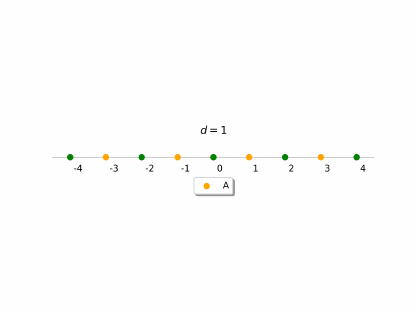}
            \caption{d = 1}
            \label{fig:sub1}
        \end{subfigure}%
        \begin{subfigure}{.5\textwidth}
            \centering
            \includegraphics[width=8cm]{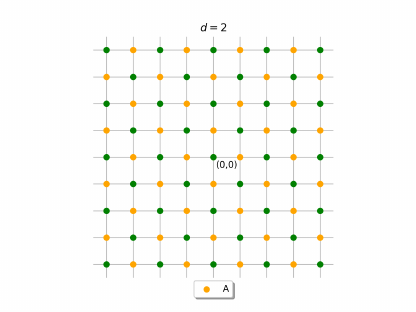}
            \caption{d = 2}
            \label{fig:sub2}
        \end{subfigure}
        \caption{Initial trap configuration of Example \ref{eg:periodic}}
        \label{fig:test}
    \end{figure}
    Let the initial trap
  configuration be such that the number of traps at $y \in \mathbb{Z}^d$ at
  time $0$ denoted by $N_y$ has the distribution $N_y \sim 
  \text{Poisson} (\nu_y)$ with
  \begin{equation}
    \nu_y = \left\{\begin{array}{l}
      \nu_1 \quad \text{if} \quad y \in A,\\
      \nu_2  \quad \text{otherwise} .
    \end{array}\right.
  \end{equation}
  where $\nu_1 \geq \nu_2$ and $\nu_1>0$. Clearly, \ref{list:A1} holds for this example and it is standard to verify that \ref{list:A2} also holds and $\bar{m} = \frac{\nu_1+\nu_2}{2}$ (see Lemma~\ref{clm:periodiccase} for complete proof).
  From Theorem~\ref{thm:annealed},
  \begin{equation}
        Z^{\gamma}_t = \begin{cases}
            \exp\left\{- \, \frac{\nu_1+\nu_2}{2}\sqrt{\frac{8}{\pi}}t^{\frac{1}{2}}(1+o(1))\right\},& \text{if } d=1,\\
        \exp\left\{- {\pi\,\frac{\nu_1+\nu_2}{2}\frac{t}{\log t}(1+o(1))}\right\}, & \text{if } d=2.
        \end{cases}
        \label{eqn:annealed periodic result}
    \end{equation}
    Therefore, asymptotically the annealed survival probability for this trap configuration is the same as that of a homogeneous system whose initial mean number of traps was equal to $\frac{\nu_1+\nu_2}{2}$. Moreover, Theorem~\ref{thm:Dt} and Theorem~\ref{thm:Nt} holds with $\bar{m} = \frac{\nu_1+\nu_2}{2}, \alpha = 1$ and $L(\cdot) \equiv 1$. The above example can be extended for any $\nu_y:\Z^d \rightarrow [0,\infty)$ which is periodic with period $N$ and satisfies \ref{list:A2}. In this case, $\bar{m}$ will be appropriate average of $\nu_y$ within the region of periodicity. 
 \end{example}

\begin{example}
\label{eg:polydecay}
Let $d=1$ and for $y\in \Z$ define
   $$\nu_y = \frac{|y|^\beta}{1+|y|^\beta}\quad \text{ for some }\beta>1.$$
\begin{figure}[H]
\centering
\begin{tikzpicture}
\begin{axis}[
    name=myaxis,
    axis lines=middle,
    every axis x line/.append style={<->},
    every axis y line/.append style={-},
    xlabel={$y$},
    ylabel={$\nu_y$},
    xmin=-10, xmax=10,
    ymin=0, ymax=1.2, 
    samples=200,
    domain=-10:10,
    width=10cm,
    height=6cm,
    xtick={-10,-8,-6,-4,-2,0,2,4,6,8,10},
    ytick={0,1},
]
\addplot[vrb, thick, dotted] 
  {(abs(x)^2)/(1+abs(x)^2)};

\addplot[dashed, black, domain=-10:10] {1}
    node[pos=0.9, above] {$\bar{m} = 1$}; 

\addplot[
    only marks,
    mark=*,
    mark size=2pt,
    vrdg
] coordinates {
    (-10,{100/101})
    (-9,{81/82})
    (-8,{64/65})
    (-7,{49/50})
    (-6,{36/37})
    (-5,{25/26})
    (-4,{16/17})
    (-3,{9/10})
    (-2,{4/5})
    (-1,{1/2})
    (0,0)
    (1,{1/2})
    (2,{4/5})
    (3,{9/10})
    (4,{16/17})
    (5,{25/26})
    (6,{36/37})
    (7,{49/50})
    (8,{64/65})
    (9,{81/82})
    (10,{100/101})
};

\coordinate (leftaxis) at (axis cs:-10,0);
\end{axis}
\draw[-stealth] (leftaxis) -- ++(-0.01cm,0);
\end{tikzpicture}
\caption{Initial trap configuration of Example \ref{eg:polydecay} with $\beta = 2$.}
        \label{fig:test2}
\end{figure}
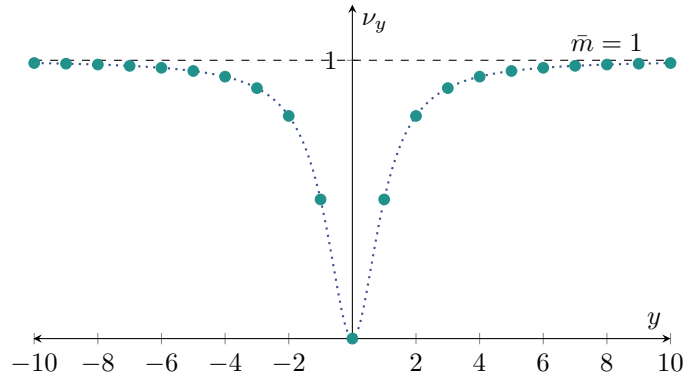

Observe that $0\le\nu_y \le 1$ for all $y\in \Z$ and hence \ref{list:A1} holds. It is standard to verify that \ref{list:A2} also holds and $\bar{m} = 1$ (see Lemma~\ref{clm:polycase} for complete proof).
Applying Theorem~\ref{thm:annealed}, we obtain
    $$Z^\gamma_t = \exp\left\{-\sqrt{\frac{8}{\pi}}t^{1/2}(1+o(1))\right\}.$$
  Moreover, Theorem~\ref{thm:Dt} and Theorem~\ref{thm:Nt} holds with $\bar{m} = 1, \alpha = 1$ and $L(\cdot) \equiv 1$.  This example can be extended to $d=2$, by setting $\nu_y = \frac{\Vert y \Vert_1^\beta}{1+{\Vert y \Vert^\beta_1}} \text{ for some }\beta>1$, where $\Vert y \Vert_1 = |y_1|+|y_2|$ and $y= (y_1,y_2)\in \Z^2$.
  \end{example}

Next, we discuss a couple of examples where the Assumption~\ref{list:A2} is not satisfied. In both instances we show that annealed survival probability decays slower than the homogeneous setting by establishing a lower bound. 
\begin{example}
\label{ex:A3holds}
Let $d=1$ and for $y\in \Z$ define
   $$\nu_y = \frac{1}{1+|y|^\beta},\quad \text{ for some }0<\beta<1.$$
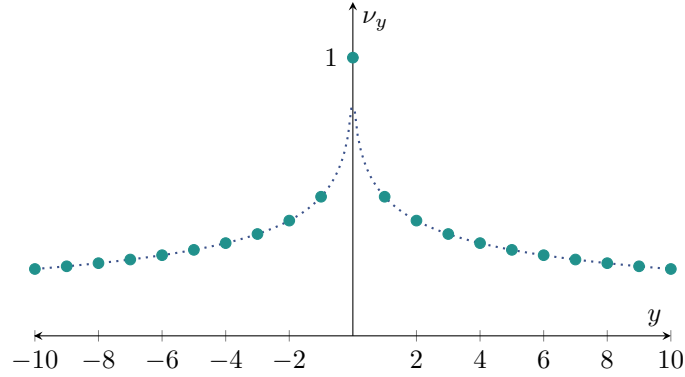
\begin{figure}[H]
\centering
\begin{tikzpicture}
\begin{axis}[
    name=myaxis,
    axis lines=middle,
    every axis x line/.append style={<->},
    every axis y line/.append style={-},
    xlabel={$y$},
    ylabel={$\nu_y$},
    xmin=-10, xmax=10,
    ymin=0, ymax=1.2,
    samples=200,
    domain=-10:10,
    width=10cm,
    height=6cm,
    xtick={-10,-8,-6,-4,-2,0,2,4,6,8,10},
    ytick={0,1},
]
\addplot[vrb, thick, dotted] 
    {1 / (1 + sqrt(abs(x)))};

\addplot[
    only marks,
    mark=*,
    mark size=2pt,
    vrdg
] coordinates {
    (-10,0.2403)
    (-9,0.2500)
    (-8,0.2612)
    (-7,0.2743)
    (-6,0.2899)
    (-5,0.3090)
    (-4,0.3333)
    (-3,0.3660)
    (-2,0.4142)
    (-1,0.5000)
    (0,1.0000)
    (1,0.5000)
    (2,0.4142)
    (3,0.3660)
    (4,0.3333)
    (5,0.3090)
    (6,0.2899)
    (7,0.2743)
    (8,0.2612)
    (9,0.2500)
    (10,0.2403)
};
\coordinate (leftaxis) at (axis cs:-10,0);
\end{axis}
\draw[-stealth] (leftaxis) -- ++(-0.01cm,0);
\end{tikzpicture}
\caption{Initial trap configuration of Example \ref{ex:A3holds} with $\beta = 0.5$.}
        \label{fig:test3}
\end{figure}
Observe that $0<\nu_y \le 1$ for all $y\in \Z$ and hence \ref{list:A1} holds. However, the condition \ref{list:A2} is not satisfied, since $\nu_y < \nu_0$ for all $y\in \Z\setminus\{0\}$. It can be verified (see Lemma~\ref{clm:A3example}) that \ref{list:A3} holds with $$\alpha = 1-\frac{\beta}{2},\,\,\, L(\cdot) \equiv 1 \,\,\,\text{ and } \,\,\bar{m} = \frac{2^{1-\frac{\beta}{2}}\Gamma\left(\frac{1-\beta}{2}\right)}{(2-\beta)\sqrt{\pi}},$$ 
 and that (see Proposition~\ref{prop:LB})  $$\liminf_{t\to \infty}\frac{1}{t^{\frac{1-\beta}{2}}} \log Z^{\gamma}_t \ge  - \sqrt{2}\bar{m}\frac{\Gamma\left(\frac{4-\beta}{2}\right)}{\Gamma\left(\frac{3-\beta}{2}\right)}.$$ Since $\frac{1-\beta}{2}< \frac{1}{2}$, the annealed survival probability decays slower than the cases where the Poisson field satisfies \ref{list:A2}, which include the homogeneous setting. 
\end{example}

\begin{example}
\label{ex:integrable}
    Consider the case where $(\nu_y)_{y\in \Z^d}$ is integrable. Let $B_r$ denote the ball of radius $r$ around the origin, $B_r = \{x\in \Z^d: \lVert x \rVert_\infty \le r\}$. Let $A_n$ be the event that $N_y = 0$ for all $y\in B_n$. Then $A_{n+1}\subseteq A_n$ and the event $A$ that no traps are placed on $\Z^d$ at time $t=0$ can be written as $A = \bigcap_{n=1}^{\infty}A_n$. Then the annealed survival probability is lower bounded by $$Z^\gamma_{t} \ge \pr(A) = \lim_{n\to \infty}\pr(A_n) = \lim_{n\to \infty} e^{-\sum_{y\in B_n}\nu_y} = e^{-\sum_{y\in \Z^d}\nu_y}.$$
    The annealed survival probability is bounded away from zero and does not even decay to zero as $t\to\infty$. 
\end{example}

\subsection{Discussion}\label{subsec:discuss}
 A key difficulty in our setting is that the trap field in the inhomogeneous setting is not time-reversible and consequently, connection to the Parabolic Anderson Model via the Feynman-Kac representation of the survival probability as done in \cite{DGRS2012} cannot be applied here. Consequently, one cannot use the subaddivity arguments like in \cite{DGRS2012} to prove the existence of annealed Lyapunov exponent. Another possible approach is to first show that the Poisson field equilibrates to a Poisson system of random walks with constant intensity $\bar{m}$. One then combines this convergence with the results from the homogeneous setting, as in \cite{DGRS2012}, to obtain the asymptotics for the annealed survival probability. However, making this approach rigorous would require a precise rate of convergence in a sufficiently strong norm, and we were unable to do this. A cumulant expansion technique was used in \cite{CG1984} for studying the homogeneous trap environment. We were able to extend this framework with some key modifications, to establish our results for the inhomogeneous Poisson environment and for the annealed survival probability of a random walk in this environment.
  
 We first describe an overview of the proof of Theorem~\ref{thm:annealed}.  The proof proceeds by obtaining matching lower and upper bounds for the annealed survival probability. The lower bound is obtained by using a classical confinement survival strategy (as outlined in \cite{DGRS2012}) in which, for a fixed time $t$, the random walker is confined in a ball of radius $R_t$ in which no traps are placed initially and no traps from outside the ball of radius $R_t$ enters this region until time $t$. By optimally choosing $R_t$, we obtain a lower bound for the survival probability in dimensions $d=1,2$. This is established in Proposition~\ref{prop:LB} below. A key ingredient in the proof of Proposition~\ref{prop:LB} is the asymptotics for the annealed survival probability when $\kappa =0$. This is proved in Proposition~\ref{lm:kappa0} where instead of using the Feynman-Kac representation, we use an adaptation of the cumulant expansion technique, mentioned above. We prove both these propositions under the weaker assumption \ref{list:A3} and both these results are of independent interest. The upper bound is obtained using the so-called \textit{Pascal principle}. It states that conditional on the random walk X, the annealed survival probability is maximized when $X \equiv 0$. The discrete time version of this result in the homogeneous setting was first proved by Moreau et al in \cite{MOBC2003, MOBC2004}, where they named it the Pascal principle, because Pascal once asserted that \emph{all misfortune of men comes from the fact that he does not stay peacefully in his room}. The continuous time version of {\textit{Pascal Principle}} was proved in \cite[Proposition 2.1]{DGRS2012} using the discrete-time version of Pascal principle. In Proposition~\ref{prop:pascal}, we provide a direct proof of the continuous time Pascal principle for the inhomogeneous trap fields under the assumption \ref{list:A2}, i.e., whenever the mean number of traps is minimal at the origin. 

The proof of Theorem~\ref{thm:Dt} and Theorem~\ref{thm:Nt} is similar in spirit to the proof of \cite[Theorem~1, Theorem~2]{CG1984}. The proof relies on the law of large numbers, central limit theorem and large deviations for dependent sequences of real random variables under a suitable hypothesis (see \ref{list:H1}). We state these well known results (see \cite{S1969,PS1975, S1978}) in Lemma~\ref{lm:largedeviations} in Section~\ref{sec:dtntproof}. We note that, in \cite{CG1984}, the authors, while studying the homogeneous case, considered occupational time functionals $D_t$ and $N_t$ for any finite subset $A\subseteq \Z^d$. We only consider the case $A = \{0\}$ and believe the same result will hold for general $A$, but we would need to modify the Assumption~\ref{list:A3} suitably.

We conclude this section with some open problems, which we plan to explore in future work. Observe that Theorem \ref{thm:annealed} shows that, under the assumptions \ref{list:A1} and \ref{list:A2}, the annealed survival probability decays sub-exponentially, as $\exp\{-C t^{1/2}\}$ in $d=1$, as $\exp\{-C' {\frac{t}{\log t}}\}$ in $d=2$. When \ref{list:A2} fails but \ref{list:A3} holds, as in Example~\ref{ex:A3holds}, we showed that the annealed survival probability is bounded below by $\exp \{-C'' t^{\frac{1-\beta}{2}}\}, 0<\beta<1$. We conjecture that the annealed survival probability in this case does decay at this slower rate. However, we were not able to prove matching upper bounds, since our current proof of Pascal Principle requires \ref{list:A2}. We believe that the weaker assumption \ref{list:A3} may indeed allow for qualitatively different asymptotic behavior of the annealed survival probability.

The entire problem remains open in $d\ge3$. The random walk is transient in dimensions bigger than or equal to three and the confinement strategy for lower bound no longer produces matching bounds. As explained in the beginning, due to lack of reversibility, new techniques are necessary to determine the correct asymptotics. One possible approach would be to use level three large deviations and apply Varadhan's lemma to establish the same. It would also be interesting to study the analogous questions for trapping problems on other infinite graphs, since decay rates depend on the behaviour of the random walk in the underlying graph. 

\paragraph{Acknowledgements}
I acknowledge support of the Department of Atomic Energy, Government of India, under project no. RTI4019.
I express my sincere gratitude to Siva Athreya, Frank den Hollander and Deepak Dhar, for their guidance, valuable suggestions throughout the course of this project  and detailed feedback during preparation of the manuscript. I would also like to thank Eleanor Archer, Alexander Drewitz and Rongfeng Sun for providing feedback on the manuscript.

\paragraph{Layout of the paper} The remainder of the paper is organised as follows. In Section~\ref{annealed}, we state the key propositions and prove Theorem~\ref{thm:annealed}. The proofs of the key propositions, Proposition~\ref{prop:pascal}, Proposition~\ref{lm:kappa0}, and Proposition~\ref{prop:LB}, are presented in Section~\ref{proofpropositions}. In Section~\ref{sec:dtntproof}, we prove Theorem~\ref{thm:Dt} and Theorem~\ref{thm:Nt}. Finally, in Section~\ref{sec:exrev}, we prove the lemmas corresponding to the examples presented in Section~\ref{subsec:examples}.

\section{Proof of Theorem~\ref{thm:annealed}}
\label{annealed}
Recall that the proof of Theorem~\ref{thm:annealed} is based on deriving matching lower and upper bounds for the annealed survival probability. We shall state three key propositions, whose proofs we will defer to Section~\ref{proofpropositions}, and complete the proof of Theorem~\ref{thm:annealed} assuming these three propositions. 

Throughout the paper, we use the shorthand $ p_t(x,y):=\pr_x^Y(Y(t)=y)$
and write $p_t(x):=p_t(0,x)$ whenever convenient. We also denote $G_t(x,y) = \int_0^t p_s(x,y) \, \dd s.$ Our first proposition is the continuous time Pascal principle for the inhomogeneous trap fields satisfying assumptions \ref{list:A1} and \ref{list:A2}.
\begin{proposition}\label{prop:pascal}{\em \bf[Pascal Principle]}
 If {\rm\ref{list:A1}} and {\ref{list:A2}} hold,
then for all piecewise constant $X:[0,t]\to \Z^d$ with $X(0)=0$ and finite number of discontinuities, we have
  \begin{equation}
   \mathbb{E}^{\xi} \left[ \exp \left\{ - \gamma \int_0^t \xi (s, X (s))
    \text{d} s \right\} \right] \leq \mathbb{E}^{\xi} \left[ \exp \left\{ -
    \gamma \int_0^t \xi (s, 0) \text{d}s \right\} \right]. \label{eqn:pascalresult}
  \end{equation}
\end{proposition} 
Proposition~\ref{prop:pascal} provides the upper bound for the annealed survival probability. Observe  that the right-hand side of \eqref{eqn:pascalresult} is precisely the annealed survival probability when the walker is stationary at the origin ($X\equiv 0$). We next prove a Lemma which shows that assumptions \ref{list:A1} and \ref{list:A2} imply \ref{list:A3} with $\alpha = 1$ and $L(\cdot)\equiv 1$. 
\begin{lemma}\label{lm:alpha1}
   If {\rm\ref{list:A1}} and {\rm \ref{list:A2}} hold, then there exists a constant $0< \bar{m}< \infty$ such that 
    $$\lim_{t\to \infty}\frac{1}{t}\int_0^t m(s,0)\, \dd s = \bar{m}.$$
\end{lemma}

\begin{proof}[Proof of Lemma~\ref{lm:alpha1}]
We first establish the existence of the limit $\lim_{t\to\infty}\frac{1}{t}\int_0^tm(s,0)\, \dd s$ using Fekete's subadditivity lemma and then show that the limit is finite and strictly positive.
For $t_1,t_2 \ge 0$, 
\begin{equation}
    \int_0^{t_1+t_2} m(s,0) \, \dd s = \int_0^{t_1} m(s,0) \, \dd s + \int_{0}^{t_2} m(s+t_1,0) \, \dd s  \label{eqn:supadd}
\end{equation}
Using the definition of $m(t,0)$, symmetry of $p(\cdot,\cdot)$,
together with the Markov property, for $s\ge 0$,
\begin{align}
    m(s+t_1 ,0 ) &= \sum_{y\in \Z^d}\nu_y  p_{s+t_1}(y,0) = \sum_{y\in \Z^d}\nu_y  \sum_{x\in \Z^d} p_{s}(y,x) p_{t_1}(x,0)\nonumber\\
    &= \sum_{x\in \Z^d} p_{t_1}(x,0) m(s,x), \label{eqn:msmarkov}
\end{align}
From \ref{list:A2} and \eqref{eqn:msmarkov} we have,
\begin{equation}
    m(s+t_1,0) \ge m(s,0) \label{eqn:monotonems0}
\end{equation}
Using \eqref{eqn:monotonems0} in \eqref{eqn:supadd},
$$\int_0^{t_1+t_2} m(s,0) \, \dd s \ge \int_0^{t_1} m(s,0) \, \dd s + \int_{0}^{t_2} m(s,0) \, \dd s.$$
Thus, $\int_0^{t} m(s,0) \, \dd s$ is superadditive in $t$. Hence by Fekete's subadditivity lemma, $\lim_{t\to\infty}\frac{1}{t}\int_0^tm(s,0)\, \dd s$ exists and is given by 
$$\bar{m} = \sup_{t>0}\frac{1}{t}\int_0^t m(s,0)\, \dd s.$$
To show finiteness, note that \ref{list:A1} implies $\nu_y \le \nu^* <\infty$ uniformly in $y$, and hence $m(t,0) \le \nu^*$ for all $t\ge 0$. Consequently, $\bar{m}\le \nu^* < \infty$. To show that $\bar{m}$ is strictly positive, observe from \eqref{eqn:monotonems0} that $m(t,0)$ is non-decreasing in $t$. Therefore, 
$$\frac{1}{t}\int_0^tm(s,0)\, \dd s \ge \frac{1}{t}\int_{t/2}^tm(s,0)\, \dd s \ge \frac{1}{2}m(t/2,0).$$
Since there exists $y\in \Z^d$ such that $\nu_y >0$ (otherwise, the trap field is identically zero), there exists $s_0 >0$ such that $m(s_0,0)>0$. Then for all $t\ge 2s_0$, 
$$\frac{1}{t}\int_0^tm(s,0)\, \dd s \ge \frac{1}{2}m(s_0,0) >0.$$
It follows that $\bar{m}>0$, completing the proof. 
\end{proof}
Given the above lemma, we shall state our next two results under the weaker assumption \ref{list:A3}.
Our next proposition gives the asymptotic behaviour of the annealed survival probability when the random walk $X$ remains stationary at the origin ($\kappa=0$).
\begin{proposition}
 \label{lm:kappa0}{\em \bf[$\kappa = 0$ solution]}
Assume that $\gamma \in (0,\infty],\kappa = 0,$ {\rm \ref{list:A1}} and {\rm \ref{list:A3}}, then
    \begin{equation}
      Z^{\gamma,0}_{t}:=  \E^\xi \left[\exp\left\{-\gamma\int_0^t \xi(s,0) \, \dd s\right\}\right] = \begin{cases}
            \exp\left\{- \sqrt{2}\, \bar{m} \frac{\Gamma(\alpha+1)}{\Gamma(\alpha+\frac{1}{2}
        )}t^{(\alpha-\frac{1}{2})}L(t)(1+o(1))\right\},& \text{if } d=1,\\
        \exp\left\{- {\pi\,\bar{m}}\frac{t^{\alpha}}{\log t}L(t)(1+o(1))\right\}, & \text{if } d=2,\\
      \exp\left\{-{\bar{m}}\frac{\gamma }{1 + \gamma G_d(0)}t^{\alpha}L(t)(1+o(1))\right\}, & \text{if } d\ge 3.
        \end{cases}
        \label{eqn:kappazerofinal}
    \end{equation}
    where  $0<\bar{m} = \lim\limits_{t\to \infty}\frac{1}{t^\alpha L(t)}\int_0^t m(s,0)\, \dd s< \infty$ and
 $G_d(0):= \int_0^t p_t(0,0)\,\dd t$ is the Green function of a simple symmetric random walk on $\Z^d$ with jump rate $1$. 
\end{proposition}
The final Proposition is the lower bound for the annealed survival probability. 
\begin{proposition}\label{prop:LB}{\em \bf[Lower bound]}
   For $\gamma \in (0,\infty],\kappa \ge 0,$ {\rm \ref{list:A1}} and {\rm\ref{list:A3}}, we have 
    \begin{equation}
    \begin{split}
       \liminf_{t\to \infty}\frac{1}{t^{(\alpha-\frac{1}{2})}L(t)} \log Z^{\gamma}_t &\ge  - \sqrt{2}\, \bar{m} \frac{\Gamma(\alpha+1)}{\Gamma(\alpha+\frac{1}{2})}, \quad \text{if } d=1,\\
       \liminf_{t\to \infty}\frac{\log t}{t^{\alpha}L(t)} \log Z^{\gamma}_t &\ge  - {\pi\,\bar{m}}, \qquad \qquad \qquad \text{if } d=2.
        \label{eqn:lowerbound}
        \end{split}
    \end{equation}
    where  $0<\bar{m} = \lim\limits_{t\to \infty}\frac{1}{t^\alpha L(t)}\int_0^t m(s,0)\, \dd s< \infty$.
\end{proposition}
We are now ready to prove Theorem~\ref{thm:annealed}.

\begin{proof}[Proof of Theorem~\ref{thm:annealed}]
Under assumptions \ref{list:A1} and \ref{list:A2},  Lemma~\ref{lm:alpha1} implies that \ref{list:A3} holds with $\alpha = 1$ and $L(\cdot) \equiv 1$ for both $d=1,2$. Applying Proposition~\ref{prop:LB}  therefore yields the lower bound,
\begin{equation}
    \begin{split}
       \liminf_{t\to \infty}\frac{1}{t^{\frac{1}{2}}} \log Z^{\gamma}_t &\ge  - \bar{m} \sqrt{\frac{8}{\pi}}, \quad \text{if } d=1,\\
       \liminf_{t\to \infty}\frac{\log t}{t} \log Z^{\gamma}_t &\ge  - {\pi\,\bar{m}},\,\,\qquad \text{if } d=2.
        \label{eqn:lbfinal}
        \end{split}
    \end{equation}
On the other hand, Proposition~\ref{prop:pascal} together with
Proposition~\ref{lm:kappa0} gives the corresponding upper bounds,
\begin{equation}
    \begin{split}
       \limsup_{t\to \infty}\frac{1}{t^{\frac{1}{2}}} \log Z^{\gamma}_t &\le  - \bar{m} \sqrt{\frac{8}{\pi}}, \quad \text{if } d=1,\\
       \limsup_{t\to \infty}\frac{\log t}{t} \log Z^{\gamma}_t &\le  - {\pi\,\bar{m}},\,\,\qquad \text{if } d=2. 
        \label{eqn:ubfinal}
        \end{split}
    \end{equation}
   Combining \eqref{eqn:lbfinal} and \eqref{eqn:ubfinal}, as the upper and lower bounds coincide asymptotically as $t\to \infty$. This completes the proof of Theorem~\ref{thm:annealed}.
\end{proof}

In the next section, we prove the three propositions stated in this section.

\section{Proof of Proposition~\ref{prop:pascal}, Proposition~\ref{lm:kappa0},  and Proposition~\ref{prop:LB}}
\label{proofpropositions}
We first fix a realisation of a random walk path $X$ starting from the origin at time $0$ and introduce some notation to derive two integral identities (for the cases $\gamma \in (0,\infty)$ and $\gamma = \infty$, respectively) for the annealed survival probability conditioned on the trajectory $X$. These identities (see \eqref{eqn:Phitilderecursion} and \eqref{eqn:Phiforgammainfinity} below), obtained via the cumulant expansion method, will be used in the proofs of Proposition~\ref{lm:kappa0} and Proposition~\ref{prop:pascal}.

\textbf{Identity for $\mathbf{\gamma \in (0,\infty)}$:}  Fix $t\geq 0$ and $\gamma\in (0,\infty)$. Integrating out the Poisson field $\xi$ in the expression of annealed survival probability in \eqref{eqn:annealedSP}, we get
\begin{equation}
    Z^{\gamma, X}_t = \E^\xi\left[\exp\left\{-\gamma\int_0^t \xi(s,X(s))\, \dd s\right\}\right] = e^{-\Phi^{\gamma, X}(t)},
    \label{eqn:Xannealeddef}
\end{equation}
where $$\Phi^{\gamma, X}(t):=  \sum_{y\in \Z^d} \nu_y \left[1- \E^Y_y\left[\exp\left\{-\gamma\int_0^t \delta_0(Y(s)- X(s))\, \dd s\right\}\right]\right],$$ $\delta_0$ is the Dirac measure at $0$ and $\E^Y_y$ is the expectation with respect to a rate $1$ simple symmetric random walk on $\Z^d$.
Set 
\begin{equation}
    \phi^X_n(t):= \frac{1}{n!}\sum_{y\in \Z^d}\nu_y \E^Y_y\left[\left(\int_0^t \delta_0(Y(s)-X(s))\, \dd s\right)^n\right],
    \label{eqn:Xsmallphindefn}
\end{equation}
for $n\ge 1$. Proceeding as in \cite[Page 313-314]{S1964} and \cite{CG1984}, we can rewrite $\phi^X_n(t)$ as 
{\small$$\phi^X_n(t) = \sum_{y\in\Z^d}\nu_y \int_0^t \dd s_n \int_0^{s_n} \dd s_{n-1} \cdots \int_0^{s_2} \dd s_1 \, p_{s_1}(y, X(s_1))\,p_{s_2-s_1}(X(s_1),X(s_2))\cdots p_{s_n - s_{n-1}}(X(s_{n-1}),X(s_n)),$$}
where $p_t(x,y)$ is the transition kernel of a continuous-time simple symmetric random walk with unit jump rate. Using Fubini's theorem and the definition of $m(t,x)$ from \eqref{eqn:poissonmean}, we can interchange sum and integral to obtain
\begin{equation}
    \phi^X_n(t) = \int_0^t \dd s_n \int_0^{s_n} \dd s_{n-1} \cdots \int_0^{s_2} \dd s_1 \, m(s_1,X(s_1)) \,p_{s_2-s_1}(X(s_1),X(s_2))\cdots p_{s_n - s_{n-1}}(X(s_{n-1}),X(s_n)).
    \label{eqn:Xpsintexpansion}
\end{equation}
From \ref{list:A1} and \eqref{eqn:Xpsintexpansion}, it is easy to see that $0\le \phi^X_n(t) \le \nu^* \frac{t^n}{n!}$, \begin{equation}
        \Phi^{\gamma, X}(t) = -\sum_{n=1}^\infty \phi^X_n(t) (-\gamma)^n
         \label{eqn:XPsiseries}
\end{equation}
 and for each fixed $t\ge 0$, the map  $\gamma \mapsto \Phi^{\gamma, X}(t)$ is real analytic on $(0,\infty).$
Observe that for $n\ge 1$,
\begin{equation}
    \phi^X_n(t) = \int_{0}^t \dd s\,\tilde{\phi}^X_{n-1}(s),
    \label{eqn:psiandphirelation}
\end{equation}
where 
\begin{align}
    \tilde{\phi}^X_0(t) &= m(t,X(t))\quad \text{and }\nonumber \\
    \tilde{\phi}^X_n(t) &= \int_0^{t} \dd s_{n-1} \cdots \int_0^{s_2} \dd s_1 \, m(s_1,X(s_1)) \,p_{s_2-s_1}(X(s_1),X(s_2))\cdots p_{t - s_{n-1}}(X(s_{n-1}),X(t))
    \label{eqn:phitilde}
\end{align}
Define
\begin{equation}
  \tilde{\Phi}^{\gamma, X}(t) := \gamma\sum_{n=0}^\infty \tilde{\phi}^X_n(t)(-\gamma)^n
    \label{eqn:XPhiexpansion}
\end{equation}
Using \eqref{eqn:psiandphirelation} in \eqref{eqn:XPsiseries}, we get 
\begin{equation}
    \Phi^{\gamma, X}(t) = \int_0^t \tilde{\Phi}^{\gamma, X}(s)\, \dd s.
    \label{eqn:XPsiasintegral}
\end{equation}
Also note that ,
\begin{equation}
\tilde{\phi}^X_n(t) = \int_{0}^t \dd s\, p_{t-s}(X(s), X(t))\, \tilde{\phi}^X_{n-1}(s)\,\, \text{for } n\ge 1.
\label{eqn:tildephirec}
\end{equation}
Using this in \eqref{eqn:XPhiexpansion}, we have the following integral recursion,
\begin{equation}
   \tilde{\Phi}^{\gamma, X}(t) = \gamma \,m(t,X(t)) - \gamma \int_{0}^t \dd s\, p_{t-s}(X(s),X(t))\,\tilde{\Phi}^{\gamma, X}(s).
    \label{eqn:Phitilderecursion}
\end{equation}
\textbf{Identity for $\mathbf{\gamma = \infty}$ :} Fix $t\geq 0$. The annealed survival probability is given by 
\begin{equation}
    \begin{split}
        Z^{\infty, X}_t &= \pr^\xi\left(\xi(s,X(s)) = 0 \text{ for all } s\in [0,t]\right)\\
        &= \exp\left\{-\sum_{y\in \Z^d} \nu_y \pr^Y_y(\tau_X \le t)\right\}\\
        &=: \exp\left\{-\nu_0 - \Phi^{\infty, X}(t)\right\}
    \end{split}
\label{eqn:Xgammainfinityannealesdefn}
\end{equation}
where $\tau_X := \inf\{s\ge 0 : Y(s) = X(s)\}$ and $\Phi^{\infty, X}(t):= \sum_{y\in \Z^d\setminus \{0\}} \nu_y \pr^Y_y(\tau_X \le t)$. For $y\neq 0$ and for some $0\le s\le t$, we can write and  
\begin{align*}
    \pr^Y_y(Y(t)=X(t)) &= \pr^Y_y(\tau_X\le t, \, Y(t) = X(t))\nonumber\\
    &= \int_0^t \pr^Y_y(\tau_X \in \dd s,\, Y(t) = X(t)) \nonumber\\
    &=  \int_0^t p_{t-s}( X(s),X(t))\pr^Y_y(\tau_X \in \dd s) \quad \text{(Strong Markov property)}.
    \label{eqn:Xstrongmarkovproperty}
\end{align*}
Multiplying both sides by $\nu_y$, summing over all $y\in \Z^d\setminus\{0\}$ and using Fubini's Theorem, we get
\begin{equation}
    m(t,X(t)) = \nu_0 p_{t}(0,X(t))+  \int_{0}^t p_{t-s}(X(s),X(t))\, \tilde{\Phi}^{\infty, X}(\dd s)
    \label{eqn:Phiforgammainfinity} 
\end{equation}
where $\tilde{\Phi}^{\infty, X}(\dd t):=\sum_{y\in\Z^d\setminus \{0\}}\nu_y \pr^Y_y(\tau_X \in \dd t)$. Note that by definition, $\Phi^{\infty, X}(t) = \int_0^t \tilde{\Phi}^{\infty, X}(\dd s)$.

\subsection{Proof of Proposition~\ref{prop:pascal}}
We shall prove Proposition~\ref{prop:pascal}, by showing that the difference $\Phi^{\gamma,X}(t)- \Phi^{\gamma,0}(t)$ is non-negative for $\gamma \in (0,\infty]$. 
 We will prove the result assuming the following lemma. The proof of the lemma uses the identity \eqref{eqn:Phitilderecursion} and is provided at the end of this subsection.
\begin{lemma}
Assume \ref{list:A2}. 
\begin{enumerate}
\item[(a)] Let $\gamma\in (0,\infty)$ and $F: [0,\infty) \to \R$ be given by $F(t)= {\Phi}^{\gamma, X}(t)-{\Phi}^{\gamma, 0}(t)$. Then, $F$ is continuous and 
\begin{equation} 
 F(u) \ge \int_0^u \dd t \,\gamma e^{-\gamma(u- t)} \int_0^t \dd s\, F(s) (p_{t-s}(0,0)- p_{t-s}(0, e_1)) \quad \text{ for all $u\geq 0$,}
 \label{eqn:FTineq}
 \end{equation}
\item[(b)] Let $\gamma = \infty$ and $G: [0,\infty) \to \R$ be given by $G(t)= {\Phi}^{\infty, X}(t)-{\Phi}^{\infty, 0}(t)$. Then, $G$ is continuous and 
\begin{equation} 
   G(t) \ge  \int_0^t \dd s\, G(s) (p_{t-s}(0,0)- p_{t-s}(0, e_1))\quad \text{for all }t\geq 0,
    \label{eqn:gtgammainfinite}
 \end{equation}
 \end{enumerate}
 where $e_1 = 1$ in $d=1$ and $e_1$ is the unit vector in $d\ge 2$.
 \label{clm:gammafinite}
\end{lemma}
\begin{proof}[Proof of Proposition~\ref{prop:pascal}]
Fix a $\gamma\in (0,\infty)$. Let $F$ be as in Lemma~\ref{clm:gammafinite}. To prove Proposition~\ref{prop:pascal} for $\gamma\in (0,\infty)$, it is enough to show that $F(t) \ge 0$ for all $t\ge 0$. Suppose $t_0 := \inf\{t\ge 0: F(t) <0 \}$, then $F(t) \ge 0$ for all $t< t_0$. The inequality implies $F(t_0)\ge0$. Since $F(t)$ is continuous, we get $F(t_0) = 0$ and for every $\delta>0$, there exists $t\in (t_0,t_0+\delta)$ such that $F(t)<0$. Fix $\delta = 1$, let $t_{1}\in (t_0, t_0+1)$ be point where $F(t)$ is minimum in the interval $(t_0, t_0+1)$, then $F(t_1)<0$. Also, $F(s)\ge F(t_1)$ for all $s\in (t_0,t_0+1)$. We have from \eqref{eqn:FTineq},
{\small \begin{align*}
    F(t_1)&\ge \int_0^{t_0} \dd t \,\gamma e^{-\gamma(t_1- t)} \int_0^t \dd s\, F(s) (p_{t-s}(0,0)- p_{t-s}(0, e_1)) \\
    &\quad \quad + \int_{t_0}^{t_1}\dd t \,\gamma e^{-\gamma(t_1- t)} \int_0^t \dd s\, F(s) (p_{t-s}(0,0)- p_{t-s}(0, e_1))\\
    &\ge \int_{t_0}^{t_1}\dd t\, \gamma e^{-\gamma(t_1- t)} \left(\int_0^{t_0} \dd s\, F(s) (p_{t-s}(0,0)- p_{t-s}(0, e_1))+\int_{t_0}^{t} \dd s\, F(s) (p_{t-s}(0,0)- p_{t-s}(0, e_1))\right) \\
    &\ge\int_{t_0}^{t_1}\dd t\, \gamma e^{-\gamma(t_1- t)} \int_{t_0}^t \dd s\, F(s) (p_{t-s}(0,0)- p_{t-s}(0, e_1))\\
    &\ge F(t_{1}) \int_{t_0}^{t_1}\dd t \,\gamma e^{-\gamma(t_1- t)} \int_{t_0}^t \dd s\,  (p_{t-s}(0,0)- p_{t-s}(0, e_1)).
\end{align*}}
Since $F(t_1)<0$, this implies 
$$\int_{t_0}^{t_1}\dd t \,\gamma e^{-\gamma(t_1- t)} \int_{t_0}^t \dd s\,  (p_{t-s}(0,0)- p_{t-s}(0, e_1))\ge 1.$$
But using the fact that $p_t$ satisfies the heat equation and is the transition probability for a simple symmetric random walk, we have
\begin{align*}
  \int_{t_0}^{t_1}\dd t \,\gamma e^{-\gamma(t_1- t)} \int_{t_0}^t \dd s\,  (p_{t-s}(0,0)- p_{t-s}(0, e_1))&= \int_{t_0}^{t_1}\dd t \,\gamma e^{-\gamma(t_1- t)} \int_{t_0}^{t} \dd s\, \frac{d}{ds}p_{t-s}(0,0)\\
  &{= \int_{t_0}^{t_1}\dd t\, \gamma e^{-\gamma(t_1- t)} (1- p_{t - t_{0}}(0,0))}\\
  & \le (1-p_1(0,0))(1-e^{-\gamma}) < 1,
\end{align*}
which is a contradiction.
Therefore $F(t)\ge 0$ for all $t\ge 0$, which proves  Proposition~\ref{prop:pascal} for $\gamma\in(0,\infty)$.

Fix $\gamma = \infty$. 
Let $G$ be as in Lemma~\ref{clm:gammafinite}. We need to show, $G(t)\ge0$ for all $t\ge 0$. As $G(0)=0$, let $t_0 := \inf\{t\ge 0: G(t) <0 \}$, then $G(t) \ge 0$ for all $t< t_0$. The inequality implies $G(t_0)\ge0$. Since $G(t)$ is continuous, we get $G(t_0) = 0$ and for every $\delta>0$, there exists $t\in (t_0,t_0+\delta)$ such that $G(t)<0$. Fix $\delta =1$, let $t_{1}\in (t_0, t_0+1)$ be point where $G(t)$ is minimum in the interval $(t_0, t_0+1)$, then $G(t_{1})<0$. Also, $G(s)\ge G(t_1)$ for all $s\in (t_0,t_0+1)$. We have 
\begin{align*}
    G(t_1)&\ge {\int_0^{t_0}\dd s\, G(s)  (p_{t_{1}-s}(0,0)- p_{t_{1}-s}(0, e_1)) +} \int_{t_0}^{t_1}\dd s\, G(s)  (p_{t_1-s}(0,0)- p_{t_1-s}(0, e_1)),\\
    &\ge  G(t_1)\int_{t_0}^{t_1}\dd s\,  (p_{t_1-s}(0,0)- p_{t_1-s}(0, e_1)),\\
    &\ge {G(t_1) \int_{t_0}^{t_1} \frac{d}{ds}p_{t_1-s}(0,0)\, \dd s = G(t_1) (1-p_{t_1-t_0}(0,0))} \ge G(t_1) (1-p_{1}(0,0)),
\end{align*}
where we use the fact that $p_t$ satisfies the heat equation and is the transition probability for a simple symmetric random walk.
{Since $G(t_1)<0$, this implies 
$(1-p_1(0,0))\ge 1$
 which is a contradiction.}
Therefore $G(t)\ge 0$ for all $t\ge 0$ which proves Proposition~\ref{prop:pascal}.
\end{proof}
To complete the proof of Proposition~\ref{prop:pascal}, we prove Lemma~\ref{clm:gammafinite}.

\begin{proof}[Proof of Lemma~\ref{clm:gammafinite}(a)]
For a given random walk path $X$, from \eqref{eqn:Xannealeddef} and \eqref{eqn:XPsiasintegral}, we can write 
$$Z^{\gamma, X}_t = \exp\{-\Phi^{\gamma,X}(t)\} =  \exp\left\{-\int_{0}^t \tilde{\Phi}^{\gamma,X}(s)\, \dd s\right\}.$$
From \eqref{eqn:phitilde}, observe that
{$$\tilde{\phi}^{X}_n(t) = \frac{1}{n!}\sum_{y\in \Z^d}\nu_y \E^Y_y\left[ \1_{\{Y(t)=X(t)\}}\left(\int_0^t \delta_0(Y(s)-X(s))\, \dd s\right)^n\right]$$}
and hence, using \eqref{eqn:XPhiexpansion},
$$ \tilde{\Phi}^{\gamma,X}(t) = \gamma \sum_{y\in \Z^d}\nu_y \E^Y_y\left[\1_{\{Y(t)=X(t)\}} \exp\left\{-\gamma \int_0^t \delta_0(Y(s)-X(s))\,\dd s\right\}\right].$$
Therefore $\tilde{\Phi}^{\gamma,X}(t) \ge 0$ for all $t\ge 0$. Using the Fourier representation and the fact that $p_t$ is the transition kernel of a simple symmetric random walk, we observe that for any $x, y\in\Z^d$, $p_{t}(x,y) \le p_t(0,0)$ and from \eqref{eqn:Phitilderecursion} we have,
\begin{equation}
    \gamma \, m(t,X(t)) \le \tilde{\Phi}^{\gamma, X}(t)+ \gamma \int_0^t\dd s \, p_{t-s}(0,0)\,\tilde{\Phi}^{\gamma, X}(s).
    \label{eqn:Xtildephiineq}
\end{equation}
Similarly for $X\equiv 0$ (i.e., when the random walker is immobile),
\begin{equation}
    \gamma \, m(t, 0) = \tilde{\Phi}^{\gamma, 0}(t)+\gamma \int_0^t \dd s\, p_{t-s}(0,0) \tilde{\Phi}^{\gamma, 0}(s).
    \label{eqn:0tildephiineq}
\end{equation}
Using the condition from \ref{list:A2}, $ m(t, 0) \le  m(t, x)$ for all $t\ge 0, x\in\Z^d$ in \eqref{eqn:0tildephiineq} and \eqref{eqn:Xtildephiineq}, we get
\begin{equation} 
\tilde{\Phi}^{\gamma, X}(t)-\tilde{\Phi}^{\gamma, 0}(t) \ge -\gamma \int_0^t \dd s\, p_{t-s}(0,0) \left[\tilde{\Phi}^{\gamma, X}(s)-\tilde{\Phi}^{\gamma, 0}(s)\right].
\label{eqn:pasineq}
\end{equation}
Note that $\Phi^{\gamma, X}(t)$ is absolutely continuous and hence $\tilde{\Phi}^{\gamma, X}(t) = \frac{d}{dt}\Phi^{\gamma, X}(t)$ almost everywhere. 
Using integration by parts in \eqref{eqn:pasineq}, we get 
\begin{equation}
    \left(\tilde{\Phi}^{\gamma, X}(t)-\tilde{\Phi}^{\gamma, 0}(t) \right)\ge  - \gamma \left({\Phi}^{\gamma, X}(t)-{\Phi}^{\gamma, 0}(t)\right) + \gamma \int_0^t \dd s\, \left({\Phi}^{\gamma, X}(s)-{\Phi}^{\gamma, 0}(s)\right) \, \frac{d }{ds}p_{t-s}(0,0).
    \label{eqn:pascalinter}
\end{equation}
Let $e_j, j = 1,\ldots,d$ denote the unit vectors in $\Z^d$. Note that $p_t(0,0)$ satisfies the partial differential equation, $\frac{\partial}{\partial t}p_t(0,0) = \Delta p_t(0,0) = \frac{1}{2d} \sum_{y\sim 0}(p_t(0,y)-p_t(0,0))$. This along with the fact that $p_t(0,\pm e_j) = p_t(0,e_1)$ for all $j = 1,\ldots, d$, we can rewrite \eqref{eqn:pascalinter} as
$$\left({\Phi}^{\gamma, X}(t)-{\Phi}^{\gamma, 0}(t)\right)+  \frac{1}{\gamma}\left(\tilde{\Phi}^{\gamma, X}(t)-\tilde{\Phi}^{\gamma, 0}(t) \right) \ge  \int_0^t \dd s\, \left({\Phi}^{\gamma, X}(s)-{\Phi}^{\gamma, 0}(s)\right)\left(p_{t-s}(0,0)- p_{t-s}(0, e_1)\right), $$
Denote $F(t) = {\Phi}^{\gamma, X}(t)-{\Phi}^{\gamma, 0}(t)$ and $f(t) = \tilde{\Phi}^{\gamma, X}(t)-\tilde{\Phi}^{\gamma, 0}(t)$. The above inequality can be rewritten in terms of $F(t), f(t)$ as 
\begin{equation}
    F(t) + \frac{1}{\gamma}f(t) \ge \int_0^t \dd s\, F(s) (p_{t-s}(0,0)- p_{t-s}(0, e_1))
    \label{eqn:ftgammafinite}
\end{equation}
 and $F(0)=0 = f(0)$. Suppose $t_0 := \inf\{t\ge 0: F(t) <0 \}$, then $F(t) \ge 0$ for all $t< t_0$. Multiply \eqref{eqn:ftgammafinite} throughout by $\gamma e^{\gamma t}$ and integrate over $t$ from $0$ to some $u>0$, to get
 $$\int_0^{u} [\gamma e^{\gamma t}F(t) + e^{\gamma t}f(t)]\, \dd t\ge \int_0^u \dd t \,\gamma e^{\gamma t} \int_0^t \dd s\, F(s) (p_{t-s}(0,0)- p_{t-s}(0, e_1)).$$
 Since $f(t)$ is derivative of $F(t)$ almost everywhere, we have \eqref{eqn:FTineq}.
\end{proof}

\begin{proof}[Proof of Lemma~\ref{clm:gammafinite}(b)]
 The annealed survival probability conditioned on random walk trajectory $X$ for $\gamma=\infty$ is given by 
\begin{equation}
    \begin{split}
        Z^{\infty, X}_t = \exp\left\{-\nu_0-\Phi^{\infty, X}(t)\right\}=\exp\left\{-\nu_0 - \int_0^t\tilde{\Phi}^{\infty, X}( \dd s) \right\}.
    \end{split}
\end{equation}
Note that $\tilde{\Phi}^{\infty, X}(\{t\}) = 0$ for all $t\ge 0$ and hence $\Phi^{\infty, X}$ is continuous.  Using $p_t(x,y)\le p_t(0,0)$ for all $x,y\in\Z^d$ in \eqref{eqn:Phiforgammainfinity}, we obtain 
\begin{equation}
    m(t,X(t)) \le \nu_0 p_{t}(0,0)+  \int_{0}^t p_{t-s}(0,0)\, \tilde{\Phi}^{\infty, X}(\dd s).
    \label{eqn:Xgammainfineq}
\end{equation}
 Similarly for $X\equiv 0$,
 \begin{equation}
     m(t,X(t)) = \nu_0 p_{t}(0,0)+  \int_{0}^t p_{t-s}(0,0)\, \tilde{\Phi}^{\infty, X}(\dd s).
     \label{eqn:0gammainfineq}
 \end{equation}
Using \eqref{eqn:pascalcond} from \ref{list:A2}, $ m(t, 0) \le  m(t, x)$ for all $t\ge 0, x\in \Z^d$ in \eqref{eqn:Xgammainfineq} and \eqref{eqn:0gammainfineq}, we get
$$\int_0^t p_{t-s}(0,0) \,\tilde{\Phi}^{\infty, X}(\dd s) -\int_0^t p_{t-s}(0,0) \,\tilde{\Phi}^{\infty, 0}(\dd s) \ge 0.$$
Using the fact that $\Phi^{\gamma, X}(t) = \int_0^t \tilde{\Phi}^{\gamma, X}(\dd s)$ and integration by parts, we get
$${\Phi}^{\infty, X}(t)-{\Phi}^{\gamma, 0}(t)\ge   \int_0^t \dd s\, \left({\Phi}^{\infty, X}(s)-{\Phi}^{\infty, 0}(s)\right)\left(p_{t-s}(0,0)- p_{t-s}(0, e_1)\right). $$
Denote $G(t) = {\Phi}^{\infty, X}(t)-{\Phi}^{\infty, 0}(t)$. The above inequality can be rewritten in terms of $G(t)$ to obtain \eqref{eqn:gtgammainfinite}.
\end{proof}

\subsection{Proof of Proposition~\ref{lm:kappa0}}
The proof of Proposition~\ref{lm:kappa0} is based on taking Laplace transforms of the identities \eqref{eqn:Phitilderecursion} and \eqref{eqn:Phiforgammainfinity} with $X\equiv 0$ and analysing their behaviour near zero. An application of the local central limit theorem and a Tauberian theorem are then used to obtain the required asymptotics of the relevant transforms and yield the large-time asymptotics of $Z^{\gamma, 0}_t$.

\begin{proof}[Proof of Proposition~\ref{lm:kappa0}]
We first consider $\gamma \in (0,\infty)$. For $\kappa= 0 $ or $X\equiv 0$, using \eqref{eqn:Xannealeddef}, \eqref{eqn:XPsiasintegral}, we have
$$Z^{\gamma, 0}_t = \exp\{-\Phi^{\gamma,0}(t)\} =  \exp\left\{-\int_{0}^t \tilde{\Phi}^{\gamma,0}(s)\, \dd s\right\}$$
and the integral identity \eqref{eqn:Phitilderecursion} reduces to,
\begin{equation}
    \tilde{\Phi}^{\gamma, 0}(t) = \gamma \,m(t,0) - \gamma \int_{0}^t \dd s\, p_{t-s}(0,0)\,\tilde{\Phi}^{\gamma, 0}(s).
\label{eqn:Phirecursionforkappa0}
\end{equation}
Denote the Laplace transform of $\tilde{\Phi}^{\gamma, 0}(t), \,p_{t}(0,0) $ and $m(t,0)$ by 
\begin{equation}
    \hat{\tilde{\Phi}}^{\gamma, 0}(\lambda)= \int_0^\infty  e^{-\lambda t} \tilde{\Phi}^{\gamma,0}(t)\,\dd t\,,\quad \hat{p}(\lambda) = \int_0^\infty e^{-\lambda t} p_t(0,0)\, \dd t\,, \quad \hat{m}(\lambda) = \int_0^t  e^{-\lambda t} m(t,0)\, \dd t,
\label{eqn:Laplacekappazero}
\end{equation}
for $\lambda>0$. Taking Laplace transforms in \eqref{eqn:Phirecursionforkappa0} gives
\begin{equation}
    \hat{\tilde{\Phi}}^{\gamma,0}(\lambda) = \frac{\gamma\hat{m}(\lambda)}{1+{\gamma}\hat{p}(\lambda)}.
    \label{eqn:PhiHatexpression}
\end{equation}
We can apply the local limit theorem for the continuous-time simple symmetric random walk in dimensions $d=1,2$ (i.e., $p_t(0,0) = \left(\frac{d}{2\pi t}\right)^{d/2}(1+o(1))$ as $t\to \infty$, see \cite[Chapter~II.7, Proposition~9]{S1964}) to obtain the following asymptotics for $\hat{p}(\lambda)$ as $\lambda \downarrow 0$,
\begin{equation}
    \hat{p}(\lambda) = \begin{cases}
   \,\frac{1}{\sqrt{2\lambda}}(1+o(1)), & \text{if } d=1,\\
  \, \frac{1}{\pi}\log \left(\frac{1}{\lambda}\right)(1+o(1)), & \text{if } d=2,\\
   \,G_d(0)(1+o(1)), & \text{if } d\ge 3,
   \end{cases}
   \label{eqn:phatlambda}
\end{equation}
where $G_d(0) = \int_0^\infty p_t(0,0) \,\dd  t$. Using \ref{list:A3} and Tauberian theorem (see \cite[Chapter~XIII.5, Theorem~2]{F1971}), we have that as $\lambda \downarrow 0$,   
\begin{equation}
\hat{m}(\lambda) = \frac{\bar{m}{\Gamma(\alpha + 1)}}{\lambda^{\alpha}}L(1/\lambda)(1+o(1)).
\label{eqn:mhatlambda}
\end{equation}
  Thus, from \eqref{eqn:phatlambda} and \eqref{eqn:mhatlambda}, we get the following asymptotics for $\hat{\tilde{\Phi}}^{\gamma, 0}(\lambda)$,
$$\hat{\tilde{\Phi}}^{\gamma, 0}(\lambda) = \begin{cases}
     \, \bar{m}{\Gamma(\alpha + 1)}\frac{\sqrt{2}}{\lambda^{(\alpha -\frac{1}{2})}}L(1/\lambda)(1+o(1)), & \text{if } d=1,\\
   \,  \bar{m}{\Gamma(\alpha + 1)}\frac{\pi}{\lambda^{\alpha}\log\left(\frac{1}{\lambda}\right)}L(1/\lambda)(1+o(1)), & \text{if } d=2,\\
    \, \bar{m}{\Gamma(\alpha + 1)} \frac{\gamma}{1 + \gamma G_d(0)}\cdot \frac{1}{\lambda^{\alpha}}L(1/\lambda)(1+o(1)), & \text{if } d\ge 3.
     \end{cases}$$
Again using Tauberian theorem \cite[Chapter~XIII.5, Theorem~2]{F1971}, we have the following asymptotics for $\Phi^{\gamma,0}(t)$ as $t\to \infty$,
$$\Phi^{\gamma,0}(t) = \begin{cases}
         \bar{m}\sqrt{2} \frac{\Gamma(\alpha+1)}{\Gamma(\alpha+\frac{1}{2}
        )}t^{(\alpha-\frac{1}{2})}L(t)(1+o(1)),& \text{if } d=1,\\
         {\pi \bar{m}} \frac{t^{\alpha}}{\log t}L(t)(1+o(1)), & \text{if } d=2,\\
         {\bar{m}} \frac{\gamma}{1 + \gamma G_d(0)}t^{\alpha}L(t)(1+o(1)), & \text{if } d\ge 3.
         \end{cases}$$
which proves Proposition~\ref{lm:kappa0} for $\gamma \in (0,\infty)$.

We now consider the case $
\gamma = \infty$. For $X\equiv 0$ or $\kappa =0$, by \eqref{eqn:Xgammainfinityannealesdefn}, we have
$$
     Z^{\infty, 0}_t = \exp\{-\nu_0 - \Phi^{\infty,0}(t)\} 
        = \exp\left\{-\nu_0-\int_0^t\tilde{\Phi}^{\infty, 0}(\dd s) \right\}.
$$
The recursion \eqref{eqn:Phiforgammainfinity} for $X\equiv 0$ will be 
\begin{equation}
    m(t,0) = \nu_0 p_{t}(0,0)+  \int_{0}^t p_{t-s}(0,0)\,\tilde{\Phi}^{\infty, X}(\dd s).
\label{eqn:kappa0Phiforgammainfinity}
\end{equation}
Denote Laplace transform of $\tilde{\Phi}^{\infty, 0}$ as
\begin{equation}
\hat{\tilde{\Phi}}^{\infty, 0}(\lambda)= \int_0^\infty  e^{-\lambda t} \tilde{\Phi}^{\infty,0}(\dd t), \text{ for } \lambda>0.
\label{eqn:gammainfmeasurelaplace}
\end{equation} 
Applying Laplace transform over \eqref{eqn:kappa0Phiforgammainfinity} using \eqref{eqn:gammainfmeasurelaplace}, \eqref{eqn:Laplacekappazero} 
and solving for $\hat{\tilde{\Phi}}^{\infty,0}(\lambda)$, we obtain
\begin{equation}
    \hat{\tilde{\Phi}}^{\infty,0}(\lambda) = \frac{\hat{m}(\lambda)}{\hat{p}(\lambda)} - \nu_0.
    \label{eqn:gammainftyPhiHatexpression}
\end{equation}
From \eqref{eqn:phatlambda} and \eqref{eqn:mhatlambda}, we get the following asymptotics for $\hat{\tilde{\Phi}}^{\infty, 0}(\lambda)$.
$$\hat{\tilde{\Phi}}^{\infty, 0}(\lambda) = \begin{cases}
     \, \bar{m}\frac{\sqrt{2}}{\lambda^{(\alpha -\frac{1}{2})}}L(1/\lambda)(1+o(1)), & \text{if } d=1,\\
   \,  \bar{m}\frac{\pi }{\lambda^{\alpha}\log\left(\frac{1}{\lambda}\right)}L(1/\lambda)(1+o(1)), & \text{if } d=2,\\
    \, \bar{m} \frac{1}{G_d(0)}\cdot \frac{1}{\lambda^{\alpha}}L(1/\lambda)(1+o(1)), & \text{if } d\ge 3,
     \end{cases}$$
Again using Tauberian theorem \cite[Chapter~XIII.5, Theorem~2]{F1971}, we have the following asymptotics for $-\log Z^{\infty,0}_t$ as $t\to \infty$,
$$-\log Z^{\infty,0}_t = \nu_0+\Phi^{\infty,0}(t) = \begin{cases}
         \bar{m}\sqrt{2} \frac{t^{(\alpha-\frac{1}{2})}}{\Gamma(\alpha+\frac{1}{2}
        )}L(t)(1+o(1)),& \text{if } d=1,\\
         \frac{\bar{m}\pi}{\Gamma(\alpha + 1)} \frac{t^{\alpha}}{\log t}L(t)(1+o(1)), & \text{if } d=2,\\
         \frac{\bar{m}}{\Gamma(\alpha+1)} \frac{1}{G_d(0)}t^{\alpha}L(t)(1+o(1)), & \text{if } d\ge 3.
         \end{cases}$$
which proves Proposition~\ref{lm:kappa0} for $\gamma = \infty$. 
\end{proof}

\subsection{Proof of Proposition~\ref{prop:LB}}
The proof of Proposition~\ref{prop:LB} proceeds by a classical confinement survival strategy. We first state a lemma which provides an useful estimate relating the initial mean number of traps in ball of radius $R>1$ to the average number of traps visiting the origin. 

\begin{lemma}\label{lm:nuysum}
    Assume {\rm \ref{list:A3}}, there exists $C_1, C_2 >0$ such that for $R>1$, 
    \begin{equation}
        \sum_{y\in B_R} \nu_y \le \begin{cases}
            \frac{C_1}{R}\int_0^{4 R^2} m(s,0)\, \dd s  & \text{ if }d=1,\\
            \frac{C_2}{\sqrt{R}} \int_0^{4R^2} m(s,0)\, \dd s  & \text{ if }d=2,
        \end{cases}
        \end{equation}
    where $B_R = \{x\in \Z^d, \lVert x \rVert_\infty \le R\}$.
\end{lemma}

\begin{proof}[Proof of Lemma~\ref{lm:nuysum}]
Let $d=1,2$. Fix $R>1$. The definition \eqref{eqn:poissonmean} yields, 
 \begin{equation}
        \int_0^R m(s,0)\, \dd s = {\int_0^R \sum_{y\in \Z} \nu_y p_s(y,0) \, \dd s } \ge \int_1^R \sum_{y\in B_{\sqrt{s}}} \nu_y p_s(y,0) \, \dd s.
        \label{eqn:int0toR}
    \end{equation}
 \medskip
 For the simple symmetric random walk on $\Z^d$, there exists $c_d>0$ such that for all $t\ge1$ and $y\in B_{\sqrt{t}}$,
$
\,p_t(y,0)\ge \frac{c_d}{t^{d/2}}.
$
 Using this in \eqref{eqn:int0toR},
    \begin{align}
        \int_0^R m(s,0)\, \dd s &{\ge c_d\int_1^R \frac{1}{s^{d/2}}\sum_{y\in B_{\sqrt{s}}} \nu_y \, \dd s}
        = 2c_d \int_1^{\sqrt{R}}\frac{1}{s^{\frac{d-1}{2}}} \sum_{y\in B_s}\nu_y \, \dd s,
    \end{align}
    {where the last step is obtained by change of variables}. Since $ \sum_{y\in B_r}\nu_y $ is increasing in $r$ and rescaling the integrals we get 
$$\int_1^{{R}} \frac{1}{s^{\frac{d-1}{2}}} \sum_{y\in B_s}\nu_y \, \dd s \le  \frac{1}{2c_d}\int_0^{R^2} m(s,0)\, \dd s.$$
In dimension $d=1$, we have
$$ {R\sum_{y\in B_R}\nu_y \le \int_R^{2R} \sum_{y\in B_r}\nu_y \, \dd r =  \int_1^{{2R}} \sum_{y\in B_s}\nu_y \, \dd s -\int_1^{{R}} \sum_{y\in B_s}\nu_y \, \dd s \le \frac{1}{2c_1}\int_0^{4R^2} m(s,0)\, \dd s,}$$
which proves Lemma~\ref{lm:nuysum} in dimension $d = 1$.

\noindent In dimension $d=2$, we have
$$ {\sqrt{\frac{R}{2}}\sum_{y\in B_R}\nu_y \le \int_R^{2R} \frac{1}{\sqrt{r}} \sum_{y\in B_r}\nu_y \, \dd r =  \int_1^{{2R}} \frac{1}{\sqrt{r}} \sum_{y\in B_s}\nu_y \, \dd s -\int_1^{{R}} \frac{1}{\sqrt{r}}\sum_{y\in B_s}\nu_y \, \dd s \le \frac{1}{2c_2}\int_0^{4R^2} m(s,0)\, \dd s,}$$
which proves Lemma~\ref{lm:nuysum} for dimension $d = 2$.
\end{proof}

We are now ready to prove Proposition~\ref{prop:LB}.
\begin{proof}[Proof of Proposition~\ref{prop:LB}]
Let $B_r$ denote the ball of radius $r$ around the origin, $B_r = \{x\in \Z^d, \lVert x \rVert_\infty \le r\}$. Choose a scale function $1<< R_t << \sqrt{t}$. Let $E_t$ denote the event that $N_y = 0$ for all $y\in B_{R_t}$. Let $F_t$ denote the event that $Y^y_j(s)\not\in B_{R_t}$ for all $y\not\in B_{R_t}, \; 1\le j\le N_y$ and $s \in [0,t]$ (i.e., any walk $Y$ that starts outside the ball of radius $B_{R_t}$ does not enter $B_{R_t}$ till time $t$). Let $G_t$ denote the event that $X$ with $X(0)=0$ does not leave the ball $B_{R_t}$ before time $t$. The annealed survival probability is lower bounded by,
\begin{equation}
    \mathbb{E}^\xi[Z^\gamma_{t, \xi}] \ge \mathbb{P}(E_t \cap F_t \cap G_t) = \mathbb{P}(E_t)\mathbb{P}(F_t)\mathbb{P}(G_t).
    \label{eq12}
\end{equation}
We first estimate $\mathbb{P}(E_t)$. Using Lemma~\ref{lm:nuysum}, there exists $C_1>0$ such that
\begin{align}
    \mathbb{P}(E_t) &= \mathbb{P}(N_y = 0 \, \forall y\in B_{R_t}) = \prod_{y\in B_{R_t}}\mathbb{P}(N_y = 0) = e^{-\sum_{y\in B_{R_t}}\nu_y} \ge  \begin{cases} e^{-\frac{C_1}{R_t}\int_0^{4R_t^2}m(s,0)\, \dd s} & \text{ if } d=1,\\
   e^{- \frac{C_2}{\sqrt{R_t}} \int_0^{4R_t^2} m(s,0)\, \dd s}  & \text{ if }d=2.
   \end{cases}
     \label{eqn:ET}
\end{align}
We next estimate $\mathbb{P}(G_t)$.
By Donsker's invariance principle, if $1<<R_t<<\sqrt{t}$ as $t\to \infty$, then there exists $\beta >0$ such that for all $t$ sufficiently large
\begin{equation}
    \inf_{x\in B_{\sqrt{t}/2}} \mathbb{P}\left(X(s)\in B_{\sqrt{t}} \; \forall s\in [0,t], X(t) \in B_{\sqrt{t}/2}|X(0) = x)\right) \ge \beta.
\end{equation}
Partitioning the time interval $[0,t]$ into intervals of length $R_t^2$, we have
\begin{align}
    \mathbb{P}(G_t) &\geq \mathbb{P}\left(X(s) \in B_{R_t} \, \forall \, s \in [(i-1)R_t^2, iR_t^2], \, \text{and } X(iR_t^2) \in B_{R_t/2}, \, i = 1, 2, \dots, \lfloor t / R_t^2 \rfloor \right) \nonumber\\
& \geq \beta^{t / R_t^2} = e^{t \log \beta / R_t^2}.
\label{eqn:GT}
\end{align}
Lastly we estimate $\mathbb{P}(F_t)$.
Let $\tilde{F}_t$ denote the event that  $Y^y_j(s) \neq 0 \text{ for all } y\in \Z^d,\; 1\le j\le N_y$ and $s\in [0,t]$.  $\mathbb{P}(\tilde{F}_t)$ is the annealed survival probability when $\kappa = 0$, and $\gamma = \infty$. Let $\tau_{B_{R_t}}$ be the stopping time when $Y$ first enters $B_{R_t}$, and $\tau_0$ be the stopping time when $Y$ first visits $0$. From usual computation of $\mathbb{P}(F_t)$,
\begin{equation}
    \log \mathbb{P}(F_t) = - \sum_{y\in \Z^d \backslash B_{R_t}}\nu_y\mathbb{P}^Y_y(\tau_{B_{R_t}}\le t).
    \label{ft}
\end{equation}
Note that,
\begin{equation}
    \sum_{y\in \Z^d \backslash B_{R_t}}\nu_y \,\mathbb{P}^Y_y(\tau_{B_{R_t}}\le t) \ge \sum_{y\in \Z^d \backslash B_{R_t}}\nu_y\, \mathbb{P}^Y_y(\tau_{0}\le t) = \sum_{y\in \Z^d}\nu_y\,\mathbb{P}^Y_y(\tau_{0}\le t) - \sum_{y\in B_{R_t}}\nu_y\,\mathbb{P}^Y_y(\tau_{0}\le t).
\end{equation}
In terms of the events, 
\begin{equation}
   \log \pr(F_t)\le \log\pr(\tilde{F}_t) + \sum_{y\in B_{R_t}}\nu_y\,\mathbb{P}^Y_y(\tau_{0}\le t) \le \log\pr(\tilde{F}_t) + \frac{C_1}{R_t}\int_0^{4R_t^2} m(s,0)\,\dd s.
\end{equation}
On the other hand, for $\epsilon > 0$, we have
\begin{align*}
\sum_{y \in \mathbb{Z}^d \setminus B_{R_t}} \nu_y\,\mathbb{P}^Y_y(\tau_{B_{R_t}} \leq t) &\geq
\sum_{y \in \mathbb{Z}^d \setminus B_{R_t}}\nu_y\, \mathbb{P}^Y_y(\tau_{B_{R_t}} \leq t, \tau_0 \leq t + \epsilon t) \\
&\geq
\inf_{z \in \partial B_{R_t}} \mathbb{P}^Y_z(\tau_0 \leq \epsilon t) \sum_{y \in \mathbb{Z}^d \setminus B_{R_t}} \nu_y\, \mathbb{P}^Y_y(\tau_{B_{R_t}} \leq t),
\end{align*}
where we used the strong Markov property. Therefore
\begin{equation}
\sum_{y \in \mathbb{Z}^d \setminus B_{R_t}}\nu_y \mathbb{P}^Y_y(\tau_{B_{R_t}} \leq t) \leq
\frac{\sum_{y \in \mathbb{Z}^d} \nu_y \mathbb{P}^Y_y(\tau_0 \leq t + \epsilon t)}{
\inf_{z \in \partial B_{R_t}} \mathbb{P}^Y_z(\tau_0 \leq \epsilon t)},
\label{eq20}
\end{equation}
and hence by \eqref{ft},  and \eqref{eq20},
\begin{equation}
\log \mathbb{P}(F_t) \ge \frac{\log \mathbb{P}(\tilde{F}_{t + \epsilon t})}{
\inf_{z \in \partial B_{R_t}}\mathbb{P}^Y_z(\tau_0 \leq t)}.
\label{eq21}
\end{equation}
For $d = 1$, let $R_t = t^{1/(2\alpha+1)}$, which is by no means the unique scale appropriate. Clearly $\inf_{z \in \partial B_{R_t}} \mathbb{P}^Y_z(\tau_0 \leq \epsilon t) \to 1$ as $t \to \infty$ and  $\mathbb{P}(\tilde{F}_t)$ satisfies the asymptotics in Proposition~\ref{lm:kappa0} for $\kappa = 0$ and $\gamma = \infty$. Since $\epsilon >0$ can be made arbitrarily small, using \eqref{eq21} we obtain
$$
\log \mathbb{P}(F_t) = -  \sqrt{2}\, \bar{m} \frac{\Gamma(\alpha+1)}{\Gamma(\alpha+\frac{1}{2}
        )}t^{(\alpha-\frac{1}{2})}L(t)(1+o(1)).
$$
Using \eqref{eqn:ET} and by the choice of $R_t$ we have
$$
\log \mathbb{P}(E_t) \ge  -\frac{C_1}{t^{\frac{1}{2\alpha+1}}}\int_0^{4 t^{\frac{2}{2\alpha +1}}} m(s,0)\, \dd s 
$$
and using \ref{list:A3},
$$
\liminf_{t\to \infty}  \frac{1}{t^{\frac{2\alpha-1}{2\alpha+1}} L(t^{\frac{2}{2\alpha+1}})}\log \mathbb{P}(E_t) \ge - C_1' \bar{m} ,
$$
for some $C_1'>0$.
Also, using \eqref{eqn:GT}
$$
\log \mathbb{P}(G_t) \geq \log \beta \, t^{1/(2\alpha+1)}.
$$
Since $\alpha> 1/2$, substituting these asymptotics into \eqref{eq12} gives,
\begin{equation*}
\liminf_{t\to \infty}\frac{1}{t^{(\alpha-\frac{1}{2})}L(t)}\log \mathbb{E}^\xi [Z^\gamma_{t, \xi}] \ge -  \sqrt{2}\, \bar{m}\frac{\Gamma(\alpha+1)}{\Gamma(\alpha+\frac{1}{2}
        )},
\end{equation*}
which proves Proposition~\ref{prop:LB} for $d=1$.

For $d = 2$, let $R_t = t^{1/(2\alpha+2)}$. Then we have $\inf_{z \in \partial B_{t^{1/(2\alpha+2)}}} \mathbb{P}^Y_z(\tau_0 \leq \epsilon t) \to 1$ as $t \to \infty$. By the same argument as for $d = 1$, we have
$$
\log \mathbb{P}(F_t) =  -\bar{m} \pi \frac{t^\alpha}{\log t} L(t) (1 + o(1)).
$$
Using \eqref{eqn:ET}, \eqref{eqn:GT} and by choice of $R_t$ we have
$$
\liminf_{t\to\infty} \frac{1}{t^{\frac{4\alpha-1}{4\alpha+4}}L(t^{\frac{2}{2\alpha +2}})}\log \mathbb{P}(E_t) \ge -C_2' \bar{m} \quad \text{and} \quad \log \mathbb{P}(G_t) \geq t^{\frac{2\alpha}{2\alpha+2}} \log \beta,
$$
for some $C_2', \beta >0$.
Therefore, using \eqref{eq12} we have
\begin{equation*}
\liminf_{t\to \infty}\frac{\log t}{t^{\alpha}L(t)}\log \mathbb{E}^\xi [Z^\gamma_{t, \xi}] \ge - \bar{m}_{\rho,d} \pi \rho,
\end{equation*}
which proves Proposition~\ref{prop:LB} for $d=2$.
\end{proof}

\section{Proof of Theorem~\ref{thm:Dt} and Theorem~\ref{thm:Nt}} \label{sec:dtntproof}
In this section, we prove Theorem~\ref{thm:Dt} and Theorem~\ref{thm:Nt}, using the cumulant expansion method (see \cite{CG1984}). We make use of the following lemma based on the results of \cite{S1969,PS1975, S1978}.

Let $(W_t)_{t\ge 0}$ be random variables on a common probability space $(\Omega, \mathcal{F}, \pr)$. Let $(a_t)_{t\ge 0}$ be a positive sequence such that $a_t \to \infty$ as $t\to \infty$. Define for $t\ge 0$ and $\lambda \in \R$
$$\Psi_t(\lambda) = \frac{1}{a_t}\log \E[e^{\lambda W_t}].$$
\textbf{Hypothesis:}
\begin{enumerate}[label= (H\arabic*)]
\item \label{list:H1}
    On the interval $(\lambda_{-}, \lambda_{+})\ni 0$,
    $$\lim_{t\to \infty} \Psi_t(\lambda) = \Psi(\lambda)< \infty.$$
    Also $\Psi(\cdot)$ is strictly convex and twice differentiable on $(\lambda_-, \lambda_+)$. 
\end{enumerate}
\begin{lemma}\label{lm:largedeviations}
Assume Hypothesis {\rm\ref{list:H1}}. Let $\mu := \Psi'(0)$, $\sigma^2 = \Psi''(0)$, $\beta_- = \lim_{\lambda \to \lambda_-}\Psi'(\lambda)$ and $\beta_+ = \lim_{\lambda \to \lambda_+}\Psi'(\lambda)$. Then, the following hold:
\begin{enumerate}
    \item[(a)] (Law of large numbers).
    $$\frac{W_t}{a_t} \to \mu \quad \text{ in probability as } t\to \infty.$$
    The convergence is almost sure on $t=1,2,\ldots$ provided $\lim_{t\to \infty}a_t/\log t = \infty$.
    \item[(b)] (Central Limit Theorem). If for each t, $\Psi_t'$ is convex on $[0,\lambda_+)$, and if $\lim_{t\to \infty}\Psi_t''(0) = \sigma^2$, then $$\frac{W_t - \E[W_t]}{\sqrt{a_t}}\xrightarrow{d}\mathcal{N}(0,\sigma^2)\quad \text{as } t\to \infty$$.
    \item[(c)] (Large deviations). For all $\beta \in (\mu, \beta_+)$, 
    $$\lim_{t\to \infty}a_t^{-1}\log \pr \left(\frac{W_t}{a_t} > \beta\right) = -I(\beta)\in (-\infty,0),$$
    and for all $\beta\in (\beta_-, \mu)$, 
    $$\lim_{t\to \infty}a_t^{-1} \log \pr\left(\frac{W_t}{a_t}<\beta\right) = -I(\beta)\in (-\infty,0),$$
    where $I(\beta) = \beta\lambda_\beta - \Psi(\lambda_\beta)$, and $\lambda_\beta$ is the unique solution of $\Psi'(\lambda_\beta) = \beta$.
\end{enumerate}
\end{lemma}
\begin{proof}
See Lemma~1 and Remarks on page 547 in \cite{CG1984}.
\end{proof}

We are now ready to prove Theorem~\ref{thm:Dt}.
\begin{proof}[Proof of Theorem~\ref{thm:Dt}]
It is enough to verify hypothesis \ref{list:H1} for the process $W_t = \frac{b_t}{\bar{m}}D_t$, where 
\begin{equation}
    b_t = \begin{cases}
        \,t^{-\frac12} & \text{ if } d=1,\\
        \,\frac{1}{\log t} & \text{ if } d=2,\\
        \, 1 & \text{ if } d\ge 3.
    \end{cases}
    \label{eqn:btdefn}
\end{equation} 
Let $a_t$ be as in \eqref{eqn:defat}. Let $\Psi_t: \R \to \R$ be defined as  
\begin{equation}
\Psi_t(\lambda) := \frac{1}{a_t}\log\E^\xi\left[\exp\left\{\frac{\lambda b_t}{\bar{m}}D_t\right\}\right].
\label{eqn:Psitlambda}
\end{equation}
Integrating out the Poisson field $\xi$, we get 
\begin{align}
    \Psi_t(\lambda) = \sum_{n=1}^\infty \frac{1}{a_t}\phi^0_n(t)\left(\frac{b_t}{\bar{m}}\right)^n\lambda^n
    \label{eqn:Psitlambdaseries}
\end{align}
where $\phi^0_n(t)$ is defined as in \eqref{eqn:Xsmallphindefn} with $X\equiv 0$. {From \ref{list:A1}, since $0\le\phi^0_n(t)\le \nu^* \frac{t^n}{n!}$ for fixed $t\ge 0$}, we have that $\Psi_t$ is {well defined and} real analytic on the whole line and $\Psi'_t$ is strictly convex on $[0,\infty)$ for each $t>0$. We assume the following lemma, whose proof is deferred until after the proof of the theorem.
\begin{lemma}
Let $\Psi_t(\cdot)$ be as in \eqref{eqn:Psitlambdaseries}, for $t\geq 0$. Under Assumptions~{\rm\ref{list:A1}} and {\rm\ref{list:A3}}, the limit
$\Psi(\lambda):=\lim_{t\to\infty}\Psi_t(\lambda)$ is given by
\begin{equation}
\Psi(\lambda)= \begin{cases}
	\displaystyle
	\,\sum_{n=1}^{\infty}\frac{\Gamma(\alpha+1)}
{\Gamma\left(\frac{n+1}{2}+\alpha\right)
(\sqrt{2}\bar{m})^{n-1}}\lambda^n, & d=1,\quad \lambda\in\mathbb{R},\\[8pt]
	\displaystyle
	\,\frac{\lambda\pi\bar{m}}{\pi\bar{m}-\lambda}, & d=2,\quad |\lambda|<\pi\bar{m},\\[8pt]
	\displaystyle
	\,\frac{\lambda}{1-\frac{\lambda}{\bar{m}}G_d(0)},& d\geq3,\quad \lambda\leq\frac{\bar{m}}{G_d(0)}.
\end{cases}
\label{eqn:Psi_combined}
\end{equation}
Moreover,
\begin{equation}
\lim_{t\to\infty}\Psi_t''(0)=
	\begin{cases}
	\displaystyle
	\,\frac{\sqrt{2}\Gamma(\alpha+1)}{\Gamma\left(\alpha+\frac{3}{2}\right)\bar{m}},& d=1,\\[8pt]
	\displaystyle
	\,\frac{2}{\pi\bar{m}}, & d=2,\\[8pt]\displaystyle
	\,\frac{2G_d(0)}{\bar{m}},& d\geq3.
\end{cases}
\label{eqn:Psi_second_combined}
\end{equation}
\label{lm:DtCombined}
\end{lemma}
Let $\Psi$ be defined as in Lemma~\ref{lm:DtCombined}.

\textbf{Case $\mathbf{d=1}$: }Since the series \eqref{eqn:Psi_combined} and the one obtained from it by term by term differentiation converge uniformly in every bounded interval,
$$\Psi'(\lambda) = 2\alpha + \frac{\sqrt{2}\Gamma(\alpha+1)}{\Gamma(\alpha+\frac{3}{2})\bar{m}} \left(\alpha+\frac{1}{2}\right)\lambda+\frac{1}{\bar{m}^2}\left(\lambda - \frac{(2\alpha-1)\bar{m}^2}{\lambda}\right)\Psi(\lambda),$$
with $\Psi(0)=0$. By solving this we get the closed form expression for $\Psi$ given in Theorem~\ref{thm:Dt}. Hypothesis (H1) will hold with $\lambda_- = -\infty$ and $\lambda_+ = +\infty$.  Also, we have $$\lim_{t\to \infty}\Psi''_t(0) =\frac{\sqrt{2}\Gamma(\alpha+1)}{\Gamma(\alpha+\frac{3}{2})\bar{m}} =\Psi''(0) =: \sigma^2.$$ Finally, we will have $\mu = 1$ , $\beta_- =0$, $\beta_+ = +\infty$, so Theorem~\ref{thm:Dt} for $d=1$ holds and $I_D(\beta) = \beta \lambda_\beta - \Psi(\lambda_\beta)$,where $\lambda_\beta$ is the unique solution to $\Psi'(\lambda_\beta) = \beta$.

\textbf{Case $\mathbf{d=2}$: } Lemma~\ref{lm:DtCombined} proves Hypothesis~\ref{list:H1} with $\lambda_- = -\pi \bar{m}, \lambda_+ = \pi\bar{m}$. We have $\mu = 1$, $\beta_- = 1/4$ and $\beta_+ = +\infty$. Also, we have $$\lim_{t\to \infty}\Psi''_t(0) =\frac{2}{\pi\bar{m}} =\Psi''(0) =: \sigma^2,$$ which verifies the assumption required to show central limit theorem.
To find $I_D(\beta)$, note that $I_D(\beta) = \beta\lambda_\beta - \Psi(\lambda_\beta)$, where $\lambda_\beta$ is such that $\Psi'(\lambda_\beta) = \beta$.  A simple calculation gives $\lambda_\beta = \pi\bar{m}\left(1-\frac{1}{\sqrt{\beta}}\right)$ and thus 
$$I_D(\beta) = \pi\bar{m}(\sqrt{\beta}-1)^2.$$
Theorem~\ref{thm:Dt} is proved for $d=2$.

\textbf{Case $\mathbf{d\ge3}$: } From Lemma~\ref{lm:DtCombined}, Hypothesis~\ref{list:H1} holds for $\lambda_- = -\infty$, $\lambda_+ = \frac{\bar{m}}{G_d(0)}$, with $\mu = 1, \beta_- = 0, \beta_+=+\infty$. Also, we have $$\lim_{t\to \infty}\Psi''_t(0) =\frac{2G_d(0)}{\bar{m}} =\Psi''(0) =: \sigma^2,$$ which verifies the assumption required to show central limit theorem. Similar computation as in $d=2$ gives $I_D(\beta) = \frac{\bar{m}}{G_d(0)}(\sqrt{\beta}-1)^2$.
\end{proof}

We now prove Lemma~\ref{lm:DtCombined} to complete the proof of Theorem~\ref{thm:Dt}.
\begin{proof}[Proof of Lemma~\ref{lm:DtCombined}]
For $t\ge 0$, set $\phi_n(t) := \phi^0_n(t)$, where $\phi^0_n(t)$ is given in \eqref{eqn:Xsmallphindefn} with $X\equiv 0$. 
A change of variable in \eqref{eqn:Xpsintexpansion} gives
\begin{equation}
    \phi_n(t) = \int_0^t \dd s\, p_s(0,0)\phi_{n-1}(t-s)
    \label{eqn:phinintrec}
\end{equation}
Denote
\begin{equation}
    \Phi_t(\lambda) := \sum_{n=1}^\infty \phi_n(t)\lambda^n
    \label{eqn:Psiseries}
\end{equation}
such that 
\begin{equation}
    \Psi_t(\lambda) = \frac{1}{a_t}\Phi_t(\lambda b_t/\bar{m}).
\end{equation}
Let $\tilde{\phi}_n(t):=\tilde{\phi}^0_n(t)$ be as defined in \eqref{eqn:phitilde}
with $X\equiv 0$.
Define
\begin{equation}
  \tilde{\Phi}_t(\lambda) = \lambda\sum_{n=0}^\infty \tilde{\phi}_n(t)\,\lambda^n.
    \label{eqn:Phiexpansion}
\end{equation}
Using \eqref{eqn:psiandphirelation} in \eqref{eqn:Psiseries}, we get 
\begin{equation}
    \Phi_t(\lambda) = \int_0^t \dd s\, \tilde{\Phi}_s(\lambda).
    \label{eqn:Psiasintegral}
\end{equation}
Using \eqref{eqn:tildephirec} in \eqref{eqn:Phiexpansion}, we obtain the following integral recursion,
\begin{equation}
   \tilde{\Phi}_t(\lambda) = \lambda \,m(t,0) +\lambda \int_{0}^t \dd s\, p_{t-s}(0,0)\,\tilde{\Phi}_s(\lambda)
    \label{eqn:tildePhirecursion}
\end{equation}
And from \eqref{eqn:Psiasintegral} we have 
\begin{equation}
  {\Phi}_t(\lambda) = \lambda \int_0^t\dd s\,m(s,0) +\lambda \int_{0}^t \dd s\, p_{t-s}(0,0)\,{\Phi}_s(\lambda)
    \label{eqn:Phirecursion}
\end{equation}
We prove Lemma~\ref{lm:DtCombined} separately in the three cases $d=1$, $d=2$, $d\ge 3$. 
\paragraph{Case $\mathbf{d=1}$:} Using the local limit theorem (see \cite[Chapter~II.7, Proposition~9]{S1964}), there is $K\ge 1$ such that 
$$p_u(0,0)\le \frac{K}{\sqrt{\pi u}} \text{ for all } u > 0.$$
We next show by induction that for all $n\geq 1$,
\begin{equation}
    \frac{\phi_n(t)}{t^{(n-1)/2}}\le \frac{K^{n-1}\phi_1(t)}{\Gamma\left(\frac{n+1}{2}\right)}.
    \label{eqn:inductionbound}
\end{equation}
When $k=1$ \eqref{eqn:inductionbound} holds with equality. Also note that, $\phi_1(t) = \int_0^t m(s,0)\,\dd s$ is monotonically increasing in $t$. Assume \eqref{eqn:inductionbound} holds for $k=n$ and using \eqref{eqn:phinintrec}, 
\begin{align*}
    \frac{\phi_{n+1}(t)}{t^{n/2}} &\le \frac{K^n}{\sqrt{\pi}t^{\frac{n}{2}}\Gamma\left(\frac{n+1}{2}\right)}\int_0^t \dd s\, \frac{1}{\sqrt{t-s}}s^{(n-1)/2}\phi_1(s).\\
    &\le \frac{K^n \phi_1(t)}{\sqrt{\pi}\Gamma\left(\frac{n+1}{2}\right)}\int_0^t\dd u\, (1-u)^{-1/2}u^{(n-1)/2}\\
    &= \frac{K^n \phi_1(t)}{\Gamma\left(\frac{n+2}{2}\right)}.   \end{align*}
Thus \eqref{eqn:inductionbound} holds for $k = n+1$. Using \eqref{eqn:inductionbound} and Assumption~\ref{list:A3}, on the asymptotic behaviour of $\int_0^tm(s,0)\,\dd s$, we can show that for any $C>1$ there exists $K< \infty$ and $t_0\ge 0$ such that 
\begin{equation}
   \frac{\phi_n(t)}{t^{(n-1)/2 + \alpha}L(t)} \le \frac{C\bar{m}K^{n-1}}{\Gamma\left(\frac{n+1}{2}\right)} \quad\text{for all $t\ge t_0, n\ge 1$.}
    \label{eqn:psinbounded}
\end{equation}
Next, we will again use induction to show that  for all $n\geq 1$,
\begin{equation}
    \lim_{t\to \infty}\frac{\phi_n(t)}{t^{\frac{n-1}{2}+\alpha}L(t)} = \frac{\bar{m}\Gamma(\alpha + 1)}{\Gamma\left(\frac{n+1}{2}+\alpha\right)(\sqrt{2})^{n-1}}.
    \label{eqn:phindim1limit}
\end{equation}
Note that, \eqref{eqn:phindim1limit} trivially holds for $k=1$ from the Assumption~\ref{list:A3}. Let $\hat{\tilde{\phi}}_{n}(\eta)= \int_0^\infty e^{-\eta t} {\tilde{\phi}}_{n}(t)\,\dd t,$ for $\eta>0$, be the laplace transform of $\tilde{\phi}_n(t)$ for $n\ge 0$. Suppose \eqref{eqn:phindim1limit} holds for $k = n$, using \eqref{eqn:psiandphirelation} and Tauberian theorem (see \cite[Chapter~XIII.5, Theorem~2]{F1971}), we have
\begin{equation}
    \lim_{\eta \to 0^+}\frac{\eta^{\frac{n-1}{2}+\alpha}}{L(1/\eta)}\hat{\tilde{\phi}}_{n-1}(\eta) = \bar{m}\Gamma(\alpha+1)\left(\frac{1}{\sqrt{2}}\right)^{n-1}.
    \label{eqn:phitildelapasym}
\end{equation}
Taking Laplace transform of \eqref{eqn:tildephirec} for $X\equiv 0$ on both sides, we get
$$\hat{\tilde{\phi}}_n(\eta) = \hat{p}(\eta)\hat{\tilde{\phi}}_{n-1}(\eta),$$
where $\hat{p}(\eta)$ is defined as in \eqref{eqn:Laplacekappazero}. From the local limit theorem (see \cite[Chapter~II.7, Proposition~9]{S1964}), we have $\hat{p}(\eta) = \frac{1}{\sqrt{2\eta}}(1+o(1))$ as $\eta\to 0^+$. This along with \eqref{eqn:phitildelapasym}, we have
$$\lim_{\eta \to 0^+}\frac{\eta^{\frac{n}{2}+\alpha}}{L(1/\eta)}\hat{\tilde{\phi}}_{n}(\eta) = \bar{m}\Gamma(\alpha+1)\left(\frac{1}{\sqrt{2}}\right)^{n}.$$
Again using Tauberian theorem (see \cite[Chapter~XIII.5, Theorem~2]{F1971}), we get 
$$ \lim_{t\to \infty}\frac{\phi_{n+1}(t)}{t^{\frac{n}{2}+\alpha}L(t)} = \frac{\bar{m}\Gamma(\alpha + 1)}{\Gamma\left(\frac{n+2}{2}+\alpha\right)(\sqrt{2})^{n}}.$$
Thus, \eqref{eqn:phindim1limit} holds for $k = n+1$ as desired. Using \eqref{eqn:psinbounded} and \eqref{eqn:phindim1limit}, and applying dominated convergence theorem in \eqref{eqn:Psitlambdaseries}, we get \eqref{eqn:Psi_combined} for $d=1$. Moreover from  \eqref{eqn:Psitlambdaseries} and \eqref{eqn:phindim1limit}, we have
$$\Psi''_t(0) = \frac{2\phi_2(t)}{\bar{m}^2 t^{1/2 + \alpha}L(t)} \rightarrow \frac{\sqrt{2}\Gamma(\alpha+1)}{\Gamma(\alpha+\frac{3}{2})\bar{m}} \quad \text{as } t\to \infty,$$
which completes the proof of Lemma~\ref{lm:DtCombined} in $d=1$.
\paragraph{Case $\mathbf{d=2}$:} 
In dimension $d=2$, for $t\geq 0$,
$$\Psi_t(\lambda) = \frac{1}{a_t}\Phi_t\left(\frac{\lambda}{\bar{m}\log t}\right).$$
Note that $\Phi_t(\lambda)$ is non-negative and non-decreasing for $\lambda\ge 0$ and non-positive and non-increasing for $\lambda\le 0$.  We will now provide upper and lower bounds for $\limsup_{t\to \infty}\Psi_t(\lambda)$ and $\liminf_{t\to \infty}\Psi_t(\lambda),$ respectively. For the upper bound, from \eqref{eqn:Phirecursion} we have
\begin{align}
    \Phi_t(\lambda) \le \lambda\int_0^t \dd s\, m(s,0) + \lambda\Phi_t(\lambda)G_t(0,0)
    \label{eqn:d2upperbound}
\end{align}
and hence
$$\Phi_t(\lambda) \le \frac{\lambda\int_0^t \dd s\, m(s,0)}{1- \lambda G_t(0,0)}, \quad \text{if } \lambda < \frac{1}{G_t(0,0)}.$$
Therefore
\begin{equation}
\Psi_t(\lambda) \le \frac{\log t}{t^\alpha L(t)}\frac{\lambda }{\bar{m}\log t(1 - \lambda\frac{G_t(0,0)}{\bar{m}\log t})}, \quad \text{if }\lambda< \frac{\bar{m}\log t}{G_t(0,0)}.
\label{eqn:Plusineq}
\end{equation}
We use the following random walk asymptotics (see \cite[Chapter~II.7, Proposition~9]{S1964}):
\begin{equation}
\lim_{t\to \infty}\frac{1}{t}p_t(0,0) = \lim_{t\to \infty}\frac{G_t(0,0)}{\log t} = \frac{1}{\pi},
\label{eqn:d2rwlclt}
\end{equation}
where $G_t(0,0) = \int_0^t p_s(0,0)\,\dd s$.
Using \eqref{eqn:d2rwlclt} in \eqref{eqn:Plusineq}, we have
\begin{equation}
    \limsup_{t\to \infty}\Psi_t(\lambda) \le \frac{\lambda \pi \bar{m}}{\pi\bar{m} - \lambda}, \quad \text{if }\lambda< \pi\bar{m}.
    \label{eqn:d2limsup}
\end{equation}
For the lower bound, we consider $\lambda\ge 0$ and $\lambda < 0$ separately. Let $0<q<1$, $\bar{q} = 1-q$. For both $\lambda\ge 0$ and $\lambda<0$, using \eqref{eqn:Phirecursion},
\begin{align}
     {\Phi}_t(\lambda) &\ge \lambda \int_0^t\dd s\,m(s,0) +\lambda \int_{0}^{\bar{q}t} \dd s\, p_{s}(0,0)\,{\Phi}_{t-s}(\lambda)\nonumber\\
     &\ge \lambda \int_0^t\dd s\,m(s,0) +\lambda \Phi_{qt}(\lambda)G_{\bar{q}t}(0,0).\label{eqn:d2lowerbound}
\end{align}
For $\lambda\ge 0$, iterating \eqref{eqn:d2lowerbound} $N$ times, for some $N\ge 1$, we get,
$$\Phi_t(\lambda) \ge \lambda \int_0^t\dd s\,m(s,0) +\lambda \sum_{n=1}^N \lambda^n\left(\int_0^{q^n t}\dd s\, m(s,0)\right)\prod_{m=0}^{n-1}G_{\bar{q}q^m t}(0,0),$$
and hence
$$\Psi_t(\lambda)\ge \frac{\lambda}{\bar{m}t^\alpha L(t)}\int_0^t m(s,0)\dd s + \frac{\lambda}{\bar{m}t^\alpha L(t)}\sum_{n=1}^N \left(\frac{\lambda}{\bar{m}}\right)^n\left(\int_0^{q^n t}\dd s\, m(s,0)\right)\prod_{m=0}^{n-1}\frac{G_{\bar{q}q^m t}(0,0)}{\log t}.$$
Using \ref{list:A3} and \eqref{eqn:d2rwlclt}, we have
$$\liminf_{t\to \infty}\Psi_t(\lambda) \ge \lambda \sum_{n=0}^N \left(\frac{\lambda q^\alpha}{\pi\bar{m}}\right)^n.$$
We then let $N\to \infty$, followed by $q\to 1$ to conclude that for $0\le \lambda < \pi \bar{m}$,
\begin{equation}
    \liminf_{t\to \infty}\Psi_t(\lambda)\ge  \frac{\lambda \pi \bar{m}}{\pi\bar{m} - \lambda}.
    \label{eqn:d2liminfpos}
\end{equation}
For $\lambda <0$, using \eqref{eqn:d2upperbound} in \eqref{eqn:d2lowerbound}, we get 
$$\Phi_t(\lambda) \ge \lambda\int_0^t m(s,0)\,\dd s + \lambda^2 G_{\bar{q}t}(0,0)\int_0^{qt}m(s,0)\,\dd s+ \lambda^2 G_{\bar{q}t}(0,0)G_{qt}(0,0)\Phi_t(\lambda).$$
Thus we obtain
$$\Phi_t(\lambda) \ge \frac{\lambda\int_0^t m(s,0)\,\dd s + \lambda^2 G_{\bar{q}t}(0,0)\int_0^{qt}m(s,0)\,\dd s}{1- \lambda^2 G_{\bar{q}t}(0,0)G_{qt}(0,0)},$$
for $\lambda^2 > \frac{1}{G_{\bar{q}t}(0,0)G_{qt}(0,0)}$. Therefore
$$\Psi_t(\lambda) \ge \frac{\frac{\lambda}{\bar{m}t^\alpha L(t)}\int_0^t m(s,0)\,\dd s + \frac{\lambda^2}{\bar{m}^2t^\alpha L(t)}\frac{ G_{\bar{q}t}(0,0)}{\log t}\int_0^{qt}m(s,0)\,\dd s}{1- \frac{\lambda^2}{\bar{m}^2} \frac{G_{\bar{q}t}(0,0)}{\log t}\frac{G_{qt}(0,0)}{\log t}},$$
for $\lambda^2 > \frac{\bar{m}^2(\log t)^2}{G_{\bar{q}t}(0,0)G_{qt}(0,0)}$. Taking limit $t\to \infty$, using \ref{list:A3} and \eqref{eqn:d2rwlclt}, we have
$$\liminf_{t\to \infty}\Psi_t(\lambda) \ge \frac{\lambda + \frac{\lambda^2 q^\alpha}{\bar{m}\pi}}{1-\frac{\lambda^2}{(\bar{m}\pi)^2}} \quad \text{for $-\pi \bar{m}< \lambda \le 0$}.$$
Allowing $q\to 1$, we obtain
\begin{equation}
    \liminf_{t\to \infty}\Psi_t(\lambda)\ge  \frac{\lambda \pi \bar{m}}{\pi\bar{m} - \lambda}\quad \text{for $-\pi \bar{m}< \lambda \le 0$}.
    \label{eqn:d2liminfneg}
\end{equation}
Therefore from\eqref{eqn:d2limsup}, \eqref{eqn:d2liminfpos}
and \eqref{eqn:d2liminfneg} proves we have \eqref{eqn:Psi_combined} for $d=2$. By a similar analysis of $\phi_2(t)$ in $d=2$ as done for $\phi_n(t)$ in $d=1$, yields \eqref{eqn:Psi_second_combined} for $d=2$.
\paragraph{Case $\mathbf{d\ge 3}$:} We have 
$$\Psi_t(\lambda) = \frac{1}{t^\alpha L(t)}\Phi_t(\lambda/\bar{m}).$$
Denote the Laplace transform of $\tilde{\Phi}_t(\lambda), \,p_{t}(0,0) $ and $m(t,0)$ by 
\begin{equation}
    \hat{\tilde{\Phi}}_\eta(\lambda)= \int_0^\infty  e^{-\eta t} \tilde{\Phi}_t(\lambda)\,\dd t\,,\quad \hat{p}(\eta) = \int_0^\infty e^{-\eta t} p_t(0,0)\, \dd t\,, \quad \hat{m}(\eta) = \int_0^t  e^{-\eta t} m(t,0)\, \dd t,
\label{eqn:Laplace}
\end{equation}
for $\eta>0$. Applying Laplace transform over \eqref{eqn:tildePhirecursion}
and solving for $\hat{\tilde{\Phi}}_\eta(\lambda)$, we obtain
\begin{equation}
    \hat{\tilde{\Phi}}_\eta(\lambda) = \frac{\lambda\hat{m}(\eta)}{1-\lambda\hat{p}(\eta)}.
    \label{eqn:tildePhiHatexpression}
\end{equation}
From \eqref{eqn:phatlambda}, we have $\hat{p}(\eta) = G_d(0)(1+o(1))$ as $\eta\to 0^+$, where $G_d(0) = \int_0^\infty p_t(0,0)\dd t$. Using Tauberian theorem (see \cite[Chapter~XIII.5, Theorem~2]{F1971}) for the Assumption~\ref{list:A3} on $\int_0^t m(s,0)\,\dd s$ gives $\hat{m}(\eta) = \Gamma(\alpha +1)\bar{m}\eta^{-\alpha}L(1/\eta)(1+o(1))$ as $\eta \to 0^+$. Thus for $0\le\lambda\le (G_d(0))^{-1}$, we have
$$\hat{\tilde{\Phi}}_\eta(\lambda) = \frac{\lambda \bar{m}\Gamma(\alpha +1)\eta^{-\alpha}L(1/\eta)}{1-\lambda G_d(0)}(1+o(1))\quad \text{as } \eta \to 0^+.$$
Using Tauberian theorem, we have for $0\le\lambda\le (G_d(0))^{-1}$,
\begin{equation}
    \lim_{t\to \infty}\frac{\Phi_t(\lambda)}{t^{\alpha}L(t)} = \frac{\bar{m}\lambda}{1-\lambda G_d(0)}.
    \label{eqn:d3limit}
\end{equation}
Similarly for $\lambda<0$, we can follow the above procedure for $-\tilde{\Phi}_t(\lambda)$ and obtain \eqref{eqn:d3limit}. Therefore we have
$$\Psi(\lambda) = \lim_{n\to \infty}\Psi_t(\lambda) = \frac{\lambda}{1-\frac{\lambda}{\bar{m}} G_d(0)}\quad \text{ for } -\infty< \lambda \le \frac{\bar{m}}{G_d(0)}.$$
By a similar analysis of $\phi_2(t)$ in $d=3$ as done for $\phi_n(t)$ in $d=1$, yields \eqref{eqn:Psi_second_combined} for $d=3$.
\end{proof}

We conclude this section with the proof of Theorem~\ref{thm:Nt}.
\begin{proof}[Proof of Theorem~\ref{thm:Nt}]
    Denote $$\theta = \begin{cases}
        \frac{\sqrt{2}\Gamma(\alpha+1)}{\Gamma(\alpha +\frac{1}{2})} & \text{ if }d=1,\\
        \pi & \text{ if }d=2,\\
        \frac{1}{G_d(0)} & \text{ if }d\ge 3,
    \end{cases}$$
    and set $Z_t = \frac{N_t}{\kappa \bar{m}}$. Then,
    $$\Psi_t(\lambda) = \frac{1}{a_t}\log\E^\xi\left[\exp\left\{\frac{\lambda}{\theta \bar{m}}\sum_{y\in\Z^d}\sum_{1\le j\le N_y}\1_{\{\tau^{y,j}_0\le t\}}\right\}\right],$$
    where $\tau^{y,j}_0 = \inf\{s\ge 0: Y^{y}_j(s) = 0\}$ and $Y_j^y := (Y_j^y (s))_{s \geq 0}$ denotes the $j$-th continuous time simple symmetric random walk with jump rate $1$ on $\Z^d$ that starts at $y$ at time $0$. Let $\tau_0 := \inf\{s\ge 0: Y(s) = 0\}$, where $Y := (Y(s))_{s \geq 0}$ denotes a continuous time simple symmetric random walk with jump rate $1$ on $\Z^d$. Integrating out the Poisson field $\xi$, we get
    \begin{align}
        \Psi_t(\lambda) &= \frac{1}{a_t}\sum_{y\in \Z^d}\nu_y \left[\E^Y_y\left[e^{\frac{\lambda}{\kappa \bar{m}}\1_{\{\tau_0\le t\}}}\right] - 1\right],\nonumber\\
        &= \frac{(e^{\frac{\lambda}{\theta\bar{m}}}-1)}{a_t}\sum_{y\in \Z^d}\nu_y \pr^Y_y(\tau_0\le t),\nonumber\\
        &= \sum_{n=1}^\infty \frac{\psi_n(t)}{n!}\lambda^n,\nonumber
    \end{align}
where $\psi_n(t):= \frac{\sum_{y\in \Z^d}\nu_y \pr^Y_y(\tau_0\le t)}{a_t(\theta\bar{m})^n}$. Evidently, $\Psi_t'(\lambda)$ is convex. From the computation of $\kappa=0, \gamma = \infty$ solution for annealed survival probability (See proof of Proposition~\ref{lm:kappa0} for $\gamma=\infty$), we have $$\sum_{y\in \Z^d}\nu_y \pr^Y_y(\tau_0\le t) = \theta \bar{m}a_t(1+o(1)) \quad\text{ as } t\to \infty.$$
Consequently, 
$$\Psi(\lambda):= \lim_{t\to \infty}\Psi_t(\lambda) =  \theta\bar{m}(e^{\frac{\lambda}{\theta\bar{m}}}-1).$$
Therefore, Hypothesis~\ref{list:H1} holds for $\lambda_- = -\infty, \lambda_+ = +\infty$, with $\mu = 1, \sigma^2= \frac{1}{\theta\bar{m}},\beta_- = 0$ and $\beta_+ = +\infty$. Also, 
$$\lim_{t \to \infty}\Psi''_t(0) = \lim_{t\to\infty}\psi_2(t) = \frac{1}{\theta\bar{m}} = \sigma^2,$$
satisfying the additional condition required for central limit theorem in Lemma~\ref{lm:largedeviations}. A routine evaluation of $I_N(\beta)$ gives $I_N(\beta) = \theta\bar{m}(\beta\log \beta - (\beta-1))$, which proves Theorem~\ref{thm:Nt}.
\end{proof}

\section{Examples Revisited}
\label{sec:exrev}
In this section, we revisit the examples stated earlier in Section~\ref{subsec:examples} and verify that they satisfy Assumption~\ref{list:A2}. Recall that for $x\in \Z^d, t\ge 0$, 
\begin{equation*}
m(t,x) := \E[\xi(t,x)]= \sum_{y\in\mathbb{Z}^d} \nu_y\,\pr_y^Y(Y(t)=x),
\end{equation*}

 \begin{lemma}
  The Poisson means $\nu_y:\Z^d \to [0,\infty)$, given by
   \begin{equation}
    \nu_y = \left\{\begin{array}{l}
      \nu_1 \quad \text{if} \quad y \in A,\\
      \nu_2  \quad \text{otherwise},
    \end{array}\right.
  \end{equation}
  where $A = \left\{ x = (x_1, \ldots, x_d) \in \mathbb{Z}^d :\;
  \sum_{i = 1}^d x_i = 2 k + 1 \;\text{for}\;
  \text{some}\;k \in \mathbb{Z} \right\},$ $\nu_1 \geq \nu_2$ and $\nu_1>0$, 
 satisfy \ref{list:A2} and $\lim_{t\to\infty}\frac{1}{t}\int_0^t m(s,0)\,\dd s = \frac{\nu_1 + \nu_2}{2}$.
  \label{clm:periodiccase}
  \end{lemma}
\begin{proof}
We have
  \begin{equation}
    m(t,x) = \left\{\begin{array}{l}
      \nu_1 \sum_{{y \in A^c} } p_t(y,0) +
      \nu_2 \sum_{y \in A } p_t(y,0) \quad
      \text{if} \quad x \in A,\\
      \nu_1 \sum_{y \in A } p_t(y,0)+ \nu_2
      \sum_{{y \in A^c} } p_t(y,0)\quad
      \text{if} \quad x \not\in A
    \end{array}\right.
    \label{eqn4.22}
  \end{equation}
  and
  \begin{equation}
    m(t,0) =
    \nu_1 \sum_{y \in A } p_t(y,0)+ \nu_2
    \sum_{{y \in A^c} } p_t(y,0).
    \label{4.23}
  \end{equation}
  The expressions for $\sum_{{y \in A^c} } p_t(y,0)$ and $\sum_{{y \in A } } p_t(y,0)$ can
  be computed explicitly using Fourier representation of the transition kernel $p_t(y,0)$ and it is given by
  \begin{equation}
      \sum_{{y \in A^c} } p_t(y,0) = \frac{1}{2}+ \frac{1}{2}e^{-2t} \quad \text{ and }\quad 
      \sum_{{y \in A} } p_t(y,0) = \frac{1}{2}- \frac{1}{2}e^{-2t}.
      \label{eqn:partition}
  \end{equation}
  From \eqref{eqn4.22} and \eqref{4.23},
  \[ m(t,x) = m(t,0) \quad \text{whenever} \quad x \not\in A \]
  and whenever $x \in A$,
  \[ m(t,x) - m(t,0) = (\nu_1 - \nu_2)e^{-2t} \geq 0. \]
  Therefore, \ref{list:A2} holds. We are left to compute $\bar m$. Note that,
\begin{equation}
    m(t,0) =
    \frac{\nu_1 + \nu_2}{2} - \frac{\nu_1-\nu_2}{2}e^{-2t}
    \label{eqn:mt0}
  \end{equation}
  Thus, 
  $$\lim_{t\to\infty}\frac{1}{t}\int_0^t m(s,0)\, \dd s = \frac{\nu_1+\nu_2}{2}.$$
  \end{proof}

   \begin{lemma}
  The Poisson means $\nu_y:\Z \to [0,\infty)$ given by
   $$\nu_y = \frac{|y|^\beta}{1+|y|^\beta},\quad \text{ for some }\beta>1,$$
 satisfy \ref{list:A2} and and $\lim_{t\to\infty}\frac{1}{t}\int_0^t m(s,0)\,\dd s = 1$.
  \label{clm:polycase}
  \end{lemma}
\begin{proof}
 We have
    $$m(t,x) = \sum_{y\in \Z}\nu_y p_t(y,x) = 1- \sum_{y\in \Z}\frac{1}{1+|y|^\beta}p_t(y,x)\quad\text{ for } x\in \Z^d.$$
    To verify if \ref{list:A2} holds, it suffices to show that for all $x\in \Z$,
    $$\sum_{y\in \Z}\frac{1}{1+|y|^\beta}p_t(y,0) \ge \sum_{y\in \Z}\frac{1}{1+|y|^\beta}p_t(y,x).$$
    Denote $k(t,x) = \sum_{y\in \Z}h(y)p_t(y,x)$ where $h(y) = \frac{1}{1+|y|^\beta}$. By the symmetry of $h$ and the transition kernel, $k(t,x)$ is an even function of $x$. Since $h$ is an even function, non-increasing on $[0,\infty)$ and $\lim\limits_{y\to \infty}h(y) = 0$, we can write 
    $$h(y) = \sum_{r\ge \lvert y\rvert}(h(r) - h(r+1)),$$ where $h(r) - h(r+1) \ge 0 \, \text{ for } r\ge 0$.
    Now for $x\ge 0$, 
    \begin{align*}
        k(t,x) &= \sum_{y\in \Z}\sum_{r\ge 0}(h(r)-h(r+1))\1_{\{\lvert y \rvert\le r\}}p_t(0, x-y)\\
        &= \sum_{r\ge 0}(h(r)-h(r+1))\sum_{y = -r}^r p_t(0, x-y)\\
        &= \sum_{r\ge 0}(h(r)-h(r+1))\sum_{y = x-r}^{x+ r} p_t(0, y).
    \end{align*}
    Therefore,
    $$k(t,x) - k(t,x+1) = \sum_{r\ge 0}(h(r)-h(r+1))(p_t(0, x-r) - p_t(0,x+r+1)).$$
    Note that for $r\ge 0, x\ge 0$, $\lvert x-r\rvert\le |x+r+1|$. Since $p_t(0, \cdot)$ is even and non-increasing on $\N\cup \{0\}$, it follows that $p_t(0, x-r) - p_t(0,x+r+1)\ge 0$. Thus, 
    $$k(t,x) \ge k(t,x+1).$$ Because $k(t,x)$ is even, this shows that $k(t,x)$ attains maximum at $x=0$. Equivalently,
    $$\sum_{y\in \Z}\frac{1}{1+|y|^\beta}p_t(y,0) \ge \sum_{y\in \Z}\frac{1}{1+|y|^\beta}p_t(y,x),$$ for all $x\in \Z$ and hence $m(t,0)\le m(t,x)$ for all $x\in \Z$. Therefore \ref{list:A2} holds. Now,
    $$\int_0^t m(s,0)\, \dd s = t - \sum_{y\in \Z}\frac{1}{1+|y|^\beta}G_t(y,0),$$
    where $G_t(w,z) = \int_0^t p_s(w,z)\,\dd s$. For the continuous time simple random walk with jump rate $1$ on $\Z$, $\lim_{t\to \infty}\frac{G_t(x,y)}{t^{1/2}} = \sqrt{\frac{2}{\pi}}$. Since $\frac{G_t(y,0)}{t}\le 1$ and $\frac{1}{1+|y|^\beta}$ is integrable whenever $\beta>1$, by Dominated Convergence Theorem we get
    $$\lim_{t\to\infty}\frac{1}{t}\int_0^t m(s,0)\, \dd s = 1.$$
\end{proof}

\begin{lemma}
Let $d=1$. For $y\in\mathbb{Z}$, define
$$
\nu_y=\frac{1}{1+|y|^\beta},
\quad 0<\beta<1.
$$
Then $(\nu_y)_{y\in\mathbb{Z}}$ satisfies \ref{list:A3} with
$\alpha=1-\frac{\beta}{2}, L(\cdot)\equiv 1,$ and $\bar m =\frac{2^{1-\beta/2}\Gamma\left(1-\frac{\beta}{2}\right)}{(2-\beta)\sqrt{\pi}}.$
\label{clm:A3example}
\end{lemma}

\begin{proof}
It is enough to show that 
\begin{equation}
\lim_{t\to\infty}t^{\beta/2}m(t,0) =
\frac{1}{\sqrt{2^\beta\pi}}\Gamma\left(\frac{1-\beta}{2}\right).
\label{eqn:mtlimit}
\end{equation}
Indeed, since $m(t,0)$ is regularly varying with index $-\beta/2$, Karamata's theorem (see \cite{F1971}) yields
$$
\lim_{t\to\infty}\frac{1}{t^{1-\beta/2}}\int_0^t m(s,0)\,ds=
\frac{2^{1-\beta/2}\Gamma\left(1-\frac{\beta}{2}\right)}{(2-\beta)\sqrt{\pi}}.
$$
Thus, it remains to prove \eqref{eqn:mtlimit}.

By the symmetry of the transition kernel $p_t(\cdot,\cdot)$ of the random walk, we can write
$$t^{\beta/2}m(t,0)=
\mathbb{E}_0^Y\left[\frac{1}{t^{-\beta/2}+|Z_t|^\beta}\right],
$$
where
$$
Z_t:=\frac{Y_t}{\sqrt{t}},
$$
and $Y=(Y_t)_{t\geq 0}$ is a continuous-time simple symmetric random walk on $\mathbb{Z}$ with jump rate $1$. By the central limit theorem, we have
$$
Z_t\xrightarrow{d} Z, \quad \text{ as } t\to\infty,
$$
where $Z$ is a standard normal random variable. Hence, it suffices to show that
\begin{equation}
\lim_{t\to\infty}
\mathbb{E}_0^Y
\left[
\frac{1}{t^{-\beta/2}+|Z_t|^\beta}
\right]
=
\mathbb{E}\left[|Z|^{-\beta}\right].
\label{eqn:singular-limit}
\end{equation}
Since for $0<\beta<1$,
\[
\mathbb{E}\left[|Z|^{-\beta}\right]
=
\frac{1}{2^{\beta/2}\sqrt{\pi}}
\Gamma\left(\frac{1-\beta}{2}\right)
<\infty.
\]
Fix $\delta>0$. Splitting according to whether $Z_t=0$, $0<|Z_t|<\delta$, or $|Z_t|\geq\delta$, we obtain
\begin{align*}
\left|
\mathbb{E}_0^Y
\left[
\frac{1}{t^{-\beta/2}+|Z_t|^\beta}
\right]
-
\mathbb{E}\left[|Z|^{-\beta}\right]
\right|
\leq
t^{\beta/2}p_t(0,0)
+
E_1(t,\delta)
+
E_2(t,\delta)
+
E_3(t,\delta),
\end{align*}
where
\begin{align*}
E_1(t,\delta)&:=\mathbb{E}_0^Y\left[\frac{1}{t^{-\beta/2}+|Z_t|^\beta}\1_{\{0<|Z_t|<\delta\}}\right],\\
E_2(t,\delta) &:=\left|\mathbb{E}_0^Y\left[|Z_t|^{-\beta}\1_{\{|Z_t|\geq\delta\}}\right]
-\mathbb{E}\left[|Z|^{-\beta}\mathbf{1}_{\{|Z|\geq\delta\}}\right]\right| \text{ and }\\
E_3(t,\delta)&:=\mathbb{E}_0^Y\left[\frac{t^{-\beta/2}|Z_t|^\beta}{t^{-\beta/2}+|Z_t|^\beta}\1_{\{|Z_t|\geq\delta\}}\right].
\end{align*}
We first estimate $E_1(t,\delta)$. Since
$$
\frac{1}{t^{-\beta/2}+|Z_t|^\beta}\leq |Z_t|^{-\beta}\quad\text{on }\{Z_t\neq0\},
$$
we have
\begin{align*}
E_1(t,\delta)&\leq \mathbb{E}_0^Y\left[|Z_t|^{-\beta}\1_{\{0<|Z_t|<\delta\}}\right]=t^{\beta/2}\sum_{1\leq k\leq\delta\sqrt{t}}k^{-\beta}p_t(0,k).
\end{align*}
For the continuous-time simple symmetric random walk on $\mathbb{Z}$, there exists a constant $c>0$ such that
$$
p_t(0,k)\leq \frac{c}{\sqrt{t}},\quad t>0,\quad k\in\mathbb{Z}.
$$
Consequently,
\begin{align*}
E_1(t,\delta)&\leq c\,t^{\frac{\beta-1}{2}}\sum_{1\leq k\leq\delta\sqrt{t}}k^{-\beta}\leq\frac{c}{1-\beta} t^{\frac{\beta-1}{2}}(\delta\sqrt{t})^{1-\beta}=\frac{c}{1-\beta}\delta^{1-\beta}.
\end{align*}
Next, on $\{|Z_t|\geq\delta\}$,
$$
\frac{t^{-\beta/2}|Z_t|^\beta}{t^{-\beta/2}+|Z_t|^\beta}\leq\frac{t^{-\beta/2}}{\delta^{2\beta}},
$$
and therefore
$$
E_3(t,\delta)\leq \frac{1}{t^{\beta/2}\delta^{2\beta}}.
$$
It remains to control $E_2(t,\delta)$. Fix $n\in\mathbb{N}$ such that $n\geq1/\delta$, and define
$$
f_n^\delta(x) = \begin{cases}
	0, & \text{ if } x\leq\delta-\frac{1}{n},\\[1mm]
	n(x-\delta)+1, & \text{ if }\delta-\frac{1}{n}<x<\delta,\\[1mm]
	1, & \text{ is } x\geq\delta.
	\end{cases}
$$
Then
$$
E_2(t,\delta) \leq D_1(t,n,\delta) + D_2(t,n,\delta) +D_3(n,\delta)+D_4(\delta),
$$
where
\begin{align*}
D_1(t,n,\delta)&:=\left|\mathbb{E}_0^Y\left[|Z_t|^{-\beta}\left(\1_{\{|Z_t|\geq\delta\}}-f_n^\delta(|Z_t|)\right)\right]\right|,\\
D_2(t,n,\delta)&:=\left|\mathbb{E}_0^Y\left[|Z_t|^{-\beta}f_n^\delta(|Z_t|)\right]-\mathbb{E}\left[|Z|^{-\beta}f_n^\delta(|Z|)\right]\right|,\\
D_3(n,\delta)&:=\left|\mathbb{E}\left[|Z|^{-\beta}f_n^\delta(|Z|)\right]-\mathbb{E}\left[|Z|^{-\beta}\1_{\{|Z|\geq\delta\}}\right]\right|, \text{ and }\\
D_4(\delta)&:=\left|\mathbb{E}\left[|Z|^{-\beta}\1_{\{|Z|<\delta\}}\right]\right|.
\end{align*}
By the definition of $f_n^\delta$,
\begin{align*}
D_1(t,n,\delta)&=\left|\mathbb{E}_0^Y\left[\frac{n(|Z_t|-\delta)+1}{|Z_t|^\beta}\1_{\{\delta-\frac1n<|Z_t|<\delta\}}\right]\right|\\
&\leq\frac{2}{(\delta-\frac1n)^\beta}\mathbb{P}_0^Y\left(\delta-\frac1n<|Z_t|<\delta\right).
\end{align*}

For fixed $n$ and $\delta$, the function $g:\R\to [0, \infty)$
$$
g(x):=|x|^{-\beta}f_n^\delta(|x|)
$$
is bounded and continuous. Since $Z_t\xrightarrow{d}Z$, it follows that
$$
D_2(t,n,\delta)\longrightarrow0,\quad \text{ as } t\to\infty.
$$
Moreover, by the local central limit theorem (see \cite[Chapter~II.7, Proposition~9]{S1964}),
$$
t^{1/2}p_t(0,0)\longrightarrow\frac{1}{\sqrt{2\pi}}, \quad \text{ as } t\to\infty,
$$
and hence
$$
t^{\beta/2}p_t(0,0)\longrightarrow0, \quad \text{ as } t\to\infty.
$$

Taking the $\limsup$ as $t\to\infty$ therefore gives
\begin{align*}
\limsup_{t\to\infty}
\left|\mathbb{E}_0^Y\left[\frac{1}{t^{-\beta/2}+|Z_t|^\beta}\right]-\mathbb{E}[|Z|^{-\beta}]\right|&\leq\frac{c}{1-\beta}\delta^{1-\beta}
+
\frac{2}{(\delta-\frac1n)^\beta}\mathbb{P}\left(\delta-\frac1n<|Z|<\delta\right)\\
&\quad+D_3(n,\delta)+D_4(\delta).
\end{align*}

Now let $n\to\infty$. Since $|Z|$ has a continuous distribution,
$$
\mathbb{P}\left(\delta-\frac1n<|Z|<\delta\right)\longrightarrow0, \quad \text{ as } n\to \infty,
$$
and, by dominated convergence theorem,
$$
D_3(n,\delta)\longrightarrow0 , \quad \text{ as } n\to \infty.
$$

Finally, since $|Z|^{-\beta}$ is integrable for $\beta<1$,
$$
D_4(\delta)=\mathbb{E}\left[|Z|^{-\beta}\mathbf{1}_{\{|Z|<\delta\}}\right]\longrightarrow0\qquad\text{as }\delta\downarrow0
$$
and $\frac{c}{1-\beta}\delta^{1-\beta} \to 0$ as $\delta\downarrow0$, since $\beta<1$.
Therefore, letting $\delta\downarrow0$ proves \eqref{eqn:singular-limit}, and hence \eqref{eqn:mtlimit}. \end{proof}

\bibliography{references}
\bibliographystyle{alpha}

\bigskip
\noindent
\textbf{Pradeeptha R Jain}\\
International Centre for Theoretical Sciences (ICTS) - TIFR,\\
Survey No. 151, Shivakote, Hesaraghatta Hobli,\\
Bengaluru - 560 089, India.\\
\textit{Email:} pradeeptha.jain@icts.res.in\\

\end{document}